\documentclass[12pt]{article}
\usepackage{amssymb}
\usepackage{amsmath}
\usepackage{amsthm}

\usepackage{graphicx}
\usepackage{bm}

\usepackage{etoc}
\etocsettocdepth{section}
\etoctocstyle{1}{Contents}

\title{
On the theory of Gordan-Noether  \\  on  homogeneous forms with zero Hessian }
\author{
Junzo Watanabe  \\ 
Department of Mathematics, Tokai University 
Hiratsuka 259-1292, Japan \\
Email: watanabe.junzo@tokai-u.jp
}

\newtheorem{theorem}{Theorem}[section]

\newtheorem{corollary}[theorem]{Corollary}
\newtheorem{lemma}[theorem]{Lemma}
\newtheorem{proposition}[theorem]{Proposition}
\theoremstyle{definition}
\newtheorem{definition}[theorem]{Definition}
\newtheorem{notation}[theorem]{Notation}
\theoremstyle{remark}
\newtheorem{example}[theorem]{Example}
\newtheorem{remark}[theorem]{Remark} 

\def\rank{{\rm rank}}
\def\dim{{\rm dim}}
\def\im{{\rm im}}
\def\ker{{\rm ker}}

\def\deg{{\rm deg}}

\def\cB{{\cal B}}
\def\cD{{\cal D}}
\def\cI{{\cal I}}

\def\sol{{\rm Sol}}
\def\bolda{\mbox{\boldmath $a$}}

\def\boldh{\mbox{\boldmath $h$}}
\def\tinyh{\mbox{\scriptsize \boldmath $h$}}
\def\boldf{\mbox{\boldmath $f$}}

\def\boldv{\mbox{\boldmath $v$}}

\def\GL{{\rm GL}}

\def\sl2{\mathfrak{sl}_2}

\def\pa{\partial}

\def\al{\alpha}

\def\om{\omega}

\newcommand{\reallylong}{\hbox{$\hbox to .5in{\rightarrowfill}$}  }

\newcommand {\PP}{\mathbb{P}}

\def\rk{\operatorname{rank}}
\def\krdim{\operatorname{Krull\,dim}}
\def\trdeg{\operatorname{tr.\,deg}}
\def\parder#1#2{\frac{\pa #1}{\pa #2}}

\begin{document}


\def\vs{\vspace{1ex}}

\begin{center}
{\Large  On the theory of Gordan-Noether on homogeneous forms with zero Hessian (improved version)} 
\end{center}

\begin{center}
{Junzo Watanabe}\\ 
{Department of Mathematics}\\ 
{Tokai University} \\  {Hiratsuka, Kanagawa 259-1292, Japan}\\ 
 {Email:watanabe.junzo@tokai-u.jp}\\ 

\vspace{1ex}
\end{center}

\begin{center}
{Michiel de Bondt}\\ 
{Institute for Mathematics, Astrophysics and Particle Physics}\\ 
{Radboud University Nijmegen, The Netherlands}\\ 
 {Email:M.deBondt@math.ru.nl}\\ 

\vspace{1ex}
\end{center}

An earlier version of this paper, namely \cite{watanabe_2},  has been printed in Proceedings of the School of Science of Tokai University, 
Vol.\@ 49, Mar.\@ 2014. 
In this version, a 
serious error has been corrected, namely \cite[Lemma~5.2]{watanabe_2}. Lemma 5.2 has been replaced by a weaker
statement, and Proposition 6.2 has been weakened along with that. Lemma 5.2 no longer suffices
for the proof of Proposition 7.2. For that reason, a new section (Section 8) has been added
to complete the proof of Proposition 7.2. Furthermore, several trivial errors have been
corrected, and a section with new results has been added (Section~9).  

\begin{center}
 {\bf Abstract}\\ 
{We give a detailed  proof for Gordan-Noether's results  in  
``Ueber die algebraischen  Formen, deren Hesse'sche Determinante identisch verschwindet.''
C.\ Lossen has written a paper in a similar direction as the present paper, but did not provide
a proof for every result. In our paper, every result is proved. Furthermore,
our paper is independent of Lossen's paper and 
includes a considerable number of new observations.} 

\vspace{1ex}
\end{center}


\begin{bf}
\tableofcontents
\end{bf}

\section{Introduction}

In 1852 and 1859, O.\  Hesse wrote two papers in Crelle's Journal Bd.\ 42 and Bd.\ 56 in  
which he claimed that if the Hessian determinant of a homogeneous polynomial identically vanishes, 
then a variable can be eliminated by  a linear transformation of the variables. 
Unfortunately his claim is not true in general. 
In fact  Hesse's  proof was unconventional and the validity of the proof was questioned 
from the beginning \cite{GNzzz}.  Nonetheless it should have been  
easy to see that Hesse's claim is true for  binary forms as well as quadrics.    
In 1875, M.\ Pasch proved that Hesse's claim is true for ternary cubics and quaternary 
cubics   \cite{moritzPasch}. 

In 1876, P.\ Gordan and M.\ Noether \cite{GNzzz} finally established the correct 
statement which says that if a form has zero Hessian, then one variable 
can be eliminated from the form itself and its partial derivatives simultaneously
by way of a birational transformation of the variables.  
Moreover they proved in the same paper  that Hesse's claim is true, 
if the number of variables  is at most four, 
and furthermore  they determined  all homogeneous polynomials in  
five variables for which the Hessian determinant identically vanishes.  

The present paper goes beyond the necessity and desire to understand their proof.   
The Hessian of a homogeneous polynomial 
is essential to the theory of Artinian Gorenstein rings because it is used  
  with higher Hessians to determine the set of the strong Lefschetz elements in a 
 zero-dimensional Gorenstein algebra (\cite{tMjW08zzz}).
In particular, if the Hessian of a homogeneous polynomial is identically zero,  
we get a Gorenstein algebra which lacks  the  strong Lefschetz property. 
To explain this further, let $R=K[x_1, \ldots,x_n]$ be  the polynomial ring over a field $K$ 
of characteristic zero  and let 
$G \in R$ be a homogeneous polynomial.  In addition let $I \subset R$ be the ideal:  
\[I=\left\{ f(x_1, \ldots, x_n) \in R \left| 
f\left(\Big(\frac{\pa }{\pa x_1}, \ldots, \frac{\pa }{\pa x_n}\Big)G\right)=0 \right. \right\},   \]
and let $A=R/I$.   
Then $A$ is a zero-dimensional Gorenstein  graded algebra $A= \bigoplus _{i=0}^d A_i$, where 
$d= \deg\;G$.  It is easy to see that if the polynomial $G$ contains properly $n$ variables, then the partial derivatives 
$G_1, \ldots, G_n$ of $G$ are linearly independent.   
Moreover if 
$L=\xi_1x_1+ \cdots + \xi _n x_n \in R_1$ is a linear form, 
it defines a  linear map  
\[\times L ^{d-2} : A_1  \longrightarrow A_{d-1}\]   
by $A_1 \ni a \mapsto aL^{d-2} \in A_{d-1}$.  Let $M$ be the matrix for this linear map 
with respect to  the bases $\langle x_1, \ldots,  x_n\rangle$ and 
$\langle G_1, \ldots,  G_n\rangle$ for $A_1$ and $A_{d-1}$ respectively.  Then  
 $\det M$ is the Hessian of $G$ evaluated at $(\xi _1, \ldots, \xi _n)$ (up to a constant multiple). 
If $\dim A _1 \leq 4$, 
and if we take the results of \cite{GNzzz} for granted, this is to  say that  
there exists  a linear form  $L$  such that $L ^{d-2}$ is bijective.  
If $n=5$, Gordan-Noether's paper enables us to determine all homogeneous polynomials 
$G$ such that  $L ^{d-2}$ is not bijective for any choice of $\xi_1, \ldots, \xi _n$.  

Gordan-Noether's paper \cite{GNzzz} has been cited by several 
authors (\cite{singhess}, \cite{HMNW}, \cite{tMjW08zzz},  \cite{W2zzz});  however each 
time it had to be accompanied with a proviso that the result is yet to be confirmed. 
Gordan-Noether's paper is difficult to understand. 
Not only their results but also the methods  have been completely forgotten.  
Thus, it seems necessary to consider the paper from the viewpoint of contemporary algebra. 
H. Yamada~\cite{HY} devoted considerable efforts to constructing a 
modernized translation of \cite{GNzzz};  
however,  it  was not completely successful and therefore it was unpublished. 

The purpose of this paper is to give detailed proofs for most of the results that were obtained by 
Gordan-Noether in \cite{GNzzz}.  
The foundation of their theory lies in the fact that if the Hessian determinant of a homogeneous  polynomial is identically zero, then 
the polynomial satisfies a certain linear partial differential equation. 
For simplicity we  assume that the Hessian matrix $M:=(\pa ^2 f/\pa x_i \pa x_j)$ has corank one. 
Then the left null space of the matrix $M$ has dimension one over the function field $K(x)$. 
Let $(h_1, \ldots, h_n)$ be a vector with $h_j \in K(x)$ such that 
$(h_1, \ldots, h_n)M=0$.  Then it is easy to see that $f$ and its partial derivatives
satisfy the partial differential equation 
\begin{equation} \label{gordan-noether}
h_1\frac{\pa F}{\pa x_1} + h_2\frac{\pa F}{\pa x_2} + \cdots + h_n\frac{\pa F}{\pa x_n}=0. 
\end{equation} 
By clearing the denominator we may think $h_j$ are homogeneous polynomials of the same 
degree.  (Cf. Remark~\ref{uniqueness_of_g_and_h}.)   
Gordan and Noether discovered  that each coefficient $h_j$ {\it itself} of (\ref{gordan-noether}) 
satisfies the partial differential equation  (\ref{gordan-noether}). This readily proves that 
a variable can be eliminated from $f$ and its partial derivatives simultaneously by a birational 
transformation. (See Theorem~\ref{1st_main_th_of_gordan_noether}.) 
The Gordan-Noether called the functions satisfying  (\ref{gordan-noether}) ``die Functionen $\Phi$.''  
The coefficients of this partial linear differential equation 
were termed as a   
``self-vanishing system'' by Yamada~\cite{HY}.  
The solution of this type of  differential equation behaves 
as if the coefficients were constants.   
According to \cite{GNzzz}, Jacobi considered this type of differential equation and 
it is the key to understanding the Gordan and Noether theory.    

To prove that Hesse's claim is true for homogeneous polynomials for $n \leq 4$, 
it is necessary to consider the fundamental locus and the image of the rational map 
$\PP ^{n-1} \to \PP ^{n-1}$ defined by $ (x_1, \ldots, x_n) \mapsto (h_1, \ldots, h_n)$.   
From the fact that  $h_1, \ldots, h_n$ is a ``self-vanishing system,''  it follows that 
the dimension of the image is at most $n - 3$ if $n \ge 3$. This enables us to determine 
the forms with zero Hessian for $n \leq 4$. For $n = 5$, the dimension of the image
of $h_1, \ldots, h_n$ may be exactly $n - 3 = 2$, which is too large to determine 
the forms with zero Hessian in dimension $5$. But this cannot occur in the context of
forms with zero Hessian. So the dimension is at most $1$ and we can determine 
the forms with zero Hessian for $n = 5$ as well.

Gordan and Noether's idea and proof techniques appear quite new and they are interesting in 
their own right, 
and they give us a series of new problems, some of which will be investigated 
in our subsequent papers. 

\begin{center}
\vskip 0.2cm
{Historical notes}
\end{center}

\vskip 0.2cm
In 1990, H.\ Yamada wrote a paper \cite{HY}, under  Grant-in-Aid no.\@ 20022551 entitled 
``On the hypersurface defined by a form whose Hessian identically vanishes.'' 
It was unpublished because it was incomplete; however,  he defined and systematically studied 
``the self-vanishing system of polynomials,'' as he named them.  
This paper was written to finish Yamada's paper \cite{HY}.   

An earlier version of this paper was written by the first author only without 
a prior knowledge of Lossen's paper \cite{clossen}. The first author 
gave a 3 hour lecture on this subject at the workshop ``Aspects of SLP and WLP'' 
held in Hawaii Tokai International College in Honolulu in September 2012, where  
he learned that Lossen  had written a paper~\cite{clossen} in the same direction and that there were  
other related papers~\cite{ciliberto_russo_francesco_simis} and \cite{garbagnati_repetto}.    
Clearly the objectives  of Lossen's paper and this paper are identical.  However 
the methods employed are different, although both are based on the same 
source, i.e., on the paper of Gordan and Noether \cite{GNzzz}. 

The main differences with Lossen's paper are listed below.  
\begin{enumerate}
\item
Theorem~\ref{1st_main_th_of_gordan_noether}
is not explicitly written in Lossen~\cite{clossen}, which tells us  that Hesse was 
in some sense correct in his intuition, when he said that if the Hessian determinant 
vanishes, then a variable can be eliminated by a linear transformation.    
\item
In this paper, a so-called ``self-vanishing system'' is defined and studied systematically.  
\item
Our observation Proposition~\ref{yamada's_proposition_3_page38}  
considerably simplify the entire argument.
\item
In Lossen's paper~\cite{clossen}, the connection to the Lefschetz properties of Artinian rings 
is not indicated.   

\item 
In this paper it is proved that a cubic form in five variables is essentially unique, 
while in \cite{clossen}, the provided proof is incomplete.
\end{enumerate}

In 2014, the first author published a version of this paper in 
Proceedings of the School of Science of Tokai University. But Lemma 
\ref{rank_of_hessian_matrix_corrected} of that paper contains a serious error. 
The error can however be fixed, and the second author has been
added to the paper to do this. In the current version, Lemma 
\ref{rank_of_hessian_matrix_corrected} has been replaced by a weaker statement. 
Proposition \ref{yamada's_proposition_3_page38} has been replaced along with 
that, and section 8 has been added to complete the proof of Proposition 
\ref{svs_for_forms_with_zero_hessian_in_5_variables} (the new version of 
Lemma \ref{rank_of_hessian_matrix_corrected} does not suffice for this). 
Section 9 is due to the second author as well, and was inspired by the final 
remark (Remark 7.6) in the 2014 version of the paper.

The second author has written some related papers on his own. 
In \cite{qtnsh}, the forms with zero Hessian are determined for $n \leq 4$.
In \cite{hmgqt5}, using results of \cite{singhess}, the forms with zero 
Hessian are determined for $n = 5$. But the proofs diverge from the techniques
in Gordan-Noether's paper \cite{GNzzz} on some points (see also remark 
\ref{linspanrem}). Non-homogeneous polynomials with zero Hessians are 
considered in \cite{singhess} and \cite{qtnsh} as well. 
In \cite{hessmalrank}, all these results are generalized to arbitrary
dimension, with the zero Hessian condition replaced by that the Hessian 
matrix has fixed small rank (one less than the original dimension). 

The first author would like to thank H.\ Nasu and T.\ Tsukioka for insightful discussions 
for Remark~\ref{rem_by_H-Nasu}.  

\section{Notation and preliminaries}
Throughout the paper we denote by $K$ an algebraically closed field of characteristic zero. We 
denote by 
$K[x_1, x_2, \ldots, x_n]$  the polynomial ring in the variables $x_1, \ldots, x_n$ and by $K(x_1, \ldots, x_n)$ 
 the function field.  
Let $R=K[x_1, \ldots, x_n]$. 
   An element of $R$ is sometimes abbreviated as  $f(x)$
or simply as $f$.  

A {\bf system} of homogeneous polynomials (or forms) of $R$ is a vector 
$(f_1, f_2, \ldots, f_n)$ consisting 
of homogeneous polynomials $f_i \in R$ of the same degree.  A system of  forms is denoted by 
a bold face letter as $\boldf=(f_1, \ldots, f_n)$.    
To avoid triviality we will always assume that $\boldf \neq 0$. 
To indicate that a system is a vector of 
polynomials depending on the argument vector $x=(x_1, \ldots, x_n)$, we write 
$\boldf(x)$ as well as $\boldf$. 
Although a coordinate system like  $x=(x_1, \ldots, x_n)$ is a ``system of forms,''  
we do not apply the rule to use  a bold face letter to denote it. 
If $y=(y_1, y_2, \ldots, y_n)$ is a  new set of variables  and if $\boldf(x)$ is a 
system of forms in $R$, then $\boldf(y)$ obtained from $\boldf(x)$ by the 
substitution  $x_j \mapsto y_j$   
is a system of forms in $K[y_1, \ldots, y_n]$. 

We treat vectors both as  row vectors and as  column vectors, 
so if $A=(a_{ij})$ is an $n \times n$ matrix with $a_{ij} \in K$,  and if 
$x=(x_1, \ldots, x_n)$ is a vector, then 
$y=Ax$  means  that 
$y=(y_1, \ldots, y_n)$ is a vector
defined by 
\[y_i = \sum _{j=1}^n a_{ij}x_j.\]
Likewise $y=xA$ means that 
\[y_j = \sum _{i=1}^na_{ij}x_i. \]
The same rule applies to systems of polynomials  as well as coordinate systems. 
Thus, if  $\boldf = (f_1, \ldots, f_n)$ is a system of forms, then 
$\boldf '= A\boldf$ 
is a system of forms for any  $n \times n$ invertible  matrix $A$ over $K$. 

\begin{lemma} \label{basic_001}  
Let $K[x]=K[x_1, \ldots, x_n]$ be the polynomial ring and let $f(x)=f(x_1, \ldots, x_n) \in K[x]$. 
Let $A=(a_{ij}) \in \GL(n,K)$, and put 
\[(x'_1 ,x'_2, \ldots, x'_n)=A(x_1, x_2 , \ldots , x_n )
\]
Consider $(x'_1. \ldots, x'_n)$ as a new coordinate system and  
let $f'$ be the polynomial in $x'_1, \ldots, x'_n$  defined by 
\[f'(x'):=f(A^{-1}x')=f(x).\]
Then 
\[(f'_1, f'_2,  \ldots , f'_n)={}^tA^{-1}(f_1, f_2,  \ldots, f_n),\]
where $f_j = \frac{\pa f}{\pa x_j}$, and $f'_j = \frac{\pa f'}{\pa x'_j}$. 
\end{lemma}   

Proof is left to the reader. 

\begin{lemma}  \label{num_of_var_involved_1}     
Let $f(x)=f(x_1, \ldots, x_n)$ be a homogeneous polynomial of positive degree. 
Put  $f_i = \frac{\pa f}{\pa x_i}$ and  let 
$\mbox{\boldmath $f$} = (f_1, \ldots, f_n)$ be  a system of polynomials. 
If $\dim _K \; \sum _{i=1} ^n Kf_i=s$, then 
$n-s$ variables can be eliminated from $f$ by means of a linear transformation of the 
variables.  In other words, 
there exists an invertible matrix 
$A=(a_{ij}) \in \GL(n,K)$ such that,  if we let 
$x'=Ax$, then  the polynomial 
$f'(x'_1, \ldots,x'_n )=f(A^{-1}x')$ does not depend on $x'_{s+1}, x'_{s+2}, \ldots, x'_{n}$.  
\end{lemma}   

\begin{proof}
By assumption there exists an invertible matrix $B$ such that  
$B(f_1, f_2, \ldots, f_n)= (f'_1, f'_2, \ldots, f'_{s}, 0,0, \allowbreak \ldots, 0).$
Let $A$ be the matrix such that ${}^tA^{-1}=B$,  and put $x'=Ax$ and  $f'(x')=f(A^{-1}x')$. 
Then \[\frac{\pa f'}{\pa x'_j}=0 \ \mbox{ for }\ j \geq s+1, \]
 by Lemma~\ref{basic_001}. 
Thus $f'(x')$ does not depend on the variables $x'_{s+1}, x'_{s+2}, \ldots, x'_n$.   
\end{proof}

\begin{proposition} \label{num_of_var_involved_2} 
Let $\mbox{\boldmath $f$} = (f_1, \ldots, f_n)$ be a system of forms in $K[x_1, \ldots, x_n]$, where 
$f_1, \ldots, f_n$ are algebraically dependent.  Let $y=(y_1, \ldots, y_n)$  be a new set of variables 
independent of $x$.  
Let \[\phi : K[y_1, y_2, \ldots, y_n] \rightarrow K[x_1, x_2, \ldots, x_n]\] 
be the homomorphism defined by $y_i \mapsto f_i$, and   
let $g=g(y_1,\ldots, y_n)$ be a nonzero element in $\ker\;  \phi$ of the least degree. 
Put 
\[h_j=h_j(x):=\frac{\pa g}{\pa y_j}(f_1, \ldots, f_n).\]
In words, $h_j(x)$  is the polynomial obtained from $\frac{\pa g}{\pa y_j}$ by substituting 
$(f_1, \ldots, f_n)$ for $(y_1, \ldots, \allowbreak y_n)$.
Let $W=\sum _{j=1}^nKh_j$ be the vector space over $K$ spanned by 
the elements $h_1, \ldots, h_n$.   
Let $s=\dim _K \; W$.  Then 
$n-s$ variables can be eliminated from $g(y)$ by means of a linear 
transformation of the variables $y_1, \ldots, y_n$.   
\end{proposition}   
\begin{proof}
Put $V=\sum _{i=1} ^n K\frac{\pa g}{\pa y_i}$.  
We claim that  $\dim _K \; V = s$.  
It is clear that  $\dim _K V \geq s$. 
Consider the restriction  
\[\phi|_V: V \rightarrow K[x_1, x_2, \ldots, x_n]. \]
We see that 
$(\ker \; \phi) \cap V =0$ by minimality of the degree of $g$.  
Thus we have $\dim _K\; V=s$,  since  $\phi | _V$ is injective 
and $\im \, \phi|_V$ is $W$.   By the previous lemma, proof is complete. 
\end{proof}

\begin{remark} \label{non_vanishing_of_boldh}   
In the above Proposition,  it is possible that $h_j=0$ for some $j$.  
It means that $\frac{\pa g}{\pa y_j}=0$.
Hence $\boldh :=(h_1, \ldots, h_n) \neq 0$.  
If we drop the condition that $\deg \; g$ is minimal in  $\cI$, 
we may define $\boldh$ as well, but $\boldh$  can be $0$.   
\end{remark}

\section{Self-vanishing systems of polynomials} \label{self-vanishing-system}
Let  $x=(x_1, \ldots, x_n)$, $y=(y_1, \ldots, y_n)$ be two sets of  indeterminates and let 
$K(x,y)=K(x_1, \ldots, x_n, \allowbreak y_1, \ldots, y_n)$ denote the rational function field. 
We introduce the differential operator 
\[\cD _x(y) : K(x,y) \rightarrow K(x,y),\]
which is defined by 
\[\cD _{x}(y)f(x,y):=\sum _{j=1}^n y_j\frac{\pa f(x,y)}{\pa x_j}\]
for $f(x,y) \in K(x,y)$.   
For a homogeneous polynomial  $f(x)=f(x_1, \ldots, x_n) \in K[x_1, \ldots, x_n]$ and 
for $j \geq 0$,  define $f^{(j)}(x,y)$ to be the 
polynomial in $K[y_1, \ldots, y_n, x_1, \ldots, x_n]$  given by 
\[f^{(j)}(x,y) = \frac{1}{j!}\cD_x(y)^j(f(x)).\] 
It is easy to see that 
\[f^{(j)}(x,x) = \binom{d}{j}f(x), \ \  j=0,1,2,\cdots, d\] 
and
\[f^{(d)}(x,y) =f(y)\] 
 where $d=\deg \;f$.  
 
\begin{proposition}   
Let $A=(a_{ij})$ be an invertible matrix with $a_{ij} \in K$ and let 
\[x'=(x'_1, \ldots, x'_n)=A(x_1, \ldots, x_n),\]
\[y'=(y'_1, \ldots, y'_n)=A(y_1, \ldots, y_n).\]
Let $f'(x')$ be the polynomial in the coordinate $x'$ defined by 
$f'(x')= f(x) \in K[x_1, \ldots, x_n]$. 
Then we have 
\[\cD_{x'}(y')f'(x')=\cD_x(y)f(x).\]
\end{proposition}

\begin{proof}
Let $\boldf$ be the system of forms:  
\[ \boldf = \frac{\pa f(x)}{\pa x}: =\left(\frac{\pa f(x)}{\pa x_1},\frac{\pa f(x)}{\pa x_2},\ldots, \frac{\pa f(x)}{\pa x_n}     \right).  \]
Likewise, let $\boldf'$  be the system of forms: 
\[ \boldf' = \frac{\pa f'(x')}{\pa x'} =\left(\frac{\pa f'(x')}{\pa x'_1},\frac{\pa f'(x')}{\pa x'_2},\ldots, \frac{\pa f'(x')}{\pa x'_n}     \right).  \]
Then, by definition, 
\[ \cD_x(y)f(x)= y\cdot \boldf = y_1\frac{\pa f}{\pa x_1} + \cdots + y_n\frac{\pa f}{\pa x_n}.\]
By  Lemma~\ref{basic_001} we have 
\begin{eqnarray*}
\cD_x(y)f(x) = y \cdot \boldf   &  =  &  \left(A^{-1}y'\right)   \cdot   \left({}^tA  \boldf '\right)  \\
                                  &  =  &  \big( y' ({}^tA^{-1}) \big)   \cdot \big({}^t A \boldf '\big)  \\
                                  &  =  &    y' \cdot  \boldf '   = \cD_{x'}(y')f'(x'). 
\end{eqnarray*}
\end{proof}

\begin{proposition}  \label{taylor_expansion}  
For $f(x) \in K[x]$, we have  
\[f(x+ty)= \sum _{j=0}^{\infty}t^jf^{(j)}(x,y),\]
for an indeterminate  $t$.  
\end{proposition}
\begin{proof}
This is a direct consequence of the Taylor expansion.  
\end{proof}

\begin{notation}  \label{notation_of_svs}      
Let  $\boldh=(h_1(x), \ldots, h_n(x))$ be a system of homogeneous polynomials
$h_i \in K[x_1, \ldots, x_n]$.   
We define the differential operator $\cD_x(\boldh): K(x) \rightarrow K(x)$ associated to $\boldh$ by 
\[\cD _x (\boldh)f(x) = \sum _{j=1}^{n} h_j(x)\frac{\pa f(x)}{\pa x_j}=f^{(1)}(x,h(x)).\] 
Furthermore we denote by 
\[\sol(\boldh; R)\]
the set of solutions in $R \subseteq K[x]$ of the differential equation 
\[\cD_x(\boldh) f(x)=0.\]
Namely, 
\[\sol(\boldh; R) = \left\{ f(x) \in R \ \left|\   \sum _{j=1}^n h_j(x) 
\frac{\pa f(x)}{\pa x_j}=0 \right. \right\}.\]
Note that $\sol(\boldh;K[x])$ is a graded  subalgebra  of $K[x]$.
\end{notation}


\begin{definition}  \label{self_vanishing_sys}    
A system  $\boldh=(h_1(x), \ldots, h_n(x))$  of polynomials is called self-vanishing, if  $h_j(x) \in \sol(\boldh; K[x])$ for 
all $j=1,2, \ldots, n$.  In addition to it, if ${\rm GCD}(h_1, \ldots, h_n)=1$, 
we will say that $\boldh$ is a reduced self-vanishing system.
\end{definition}

\begin{example}    
A constant vector  $\boldh=(c_1, c_2,  \ldots, c_n) \in K^n$ is  obviously a self-vanishing system. 
\end{example}


\begin{example}  \label{splitting_case}   
Let $h_j \in K[x]$ be homogeneous polynomials (of the same degree).  
Suppose that $\boldh = (h_1, \ldots, h_n)$ satisfy the following conditions. 
\begin{enumerate}
\item
$h_1 = \cdots = h_r = 0, \mbox{ for some integer } r;   \ 1 \leq  r < n. $
\item
The polynomials $h_{r+1}, \ldots,  h_{n}$ do not involve the 
variables $x_{r+1}, \ldots, x_n$. 
\end{enumerate}
Then  $\boldh$ is a self-vanishing system of forms.    
\end{example}


\begin{definition}  
Let $\boldf=(f_1, \ldots, f_n)$ be a system of forms in $K[x]$.  
We denote by   ${\pa \boldf}/{\pa x_j}$ the 
system of forms: 
\[\frac{\pa \boldf}{\pa x_j}=\frac{\pa}{\pa x_j}\boldf=\left(\frac{\pa f_1}{\pa x_j}, \frac{\pa f_2}{\pa x_j}, \ldots , \frac{\pa f_n}{\pa x_j}\right).\]
This should not be confused with the notation already used:  
\[\frac{\pa f(x)}{\pa x}:=\left(\frac{\pa f}{\pa x_1}, \frac{\pa f}{\pa x_2}, \ldots, \frac{\pa f}{\pa x_n}\right).\]
\end{definition}

\begin{proposition} \label{system_associated_to_alg_dep_system_is_svs} 
Let  $\boldf:=(f_1, f_2, \ldots, f_n) $ be  a system of forms in $K[x]$, in which the components are 
algebraically dependent. 
Let $y=(y_1, \ldots, y_n)$ be a coordinate system  algebraically independent of $x$. 
Let $\phi : K[y] \rightarrow K[x]$ be the homomorphism defined by 
\[y_j \mapsto f_j, (j=1, 2, \ldots, n).\] 
Let $g=g(y) \in \ker \, \phi$ be a non-zero homogeneous polynomial of the least degree 
in $\ker\, \phi$.    
As in Proposition~\ref{num_of_var_involved_2},  define $h_j \in K[x]$ by  
\[h_j=\frac{\pa g}{\pa y_j}(f_1, \ldots, f_n).\]
Let $\boldh=(h_1, \ldots, h_n)$.  
Then 
\begin{enumerate}
\item[{\rm (a)}] 
$\boldh$ is a syzygy of $\boldf$.   
\item[{\rm (b)}]
$\boldh$ is a syzygy  of $\pa \boldf/\pa x_j$ for every $j=1,2, \ldots, n$.   
\end{enumerate}
\end{proposition}

\begin{proof}
Since $g=g(y)$ is homogeneous, we have 
\[\frac{\pa g}{\pa y_1}y_1 + \frac{\pa g}{\pa y_2}y_2 + \cdots +\frac{\pa g}{\pa y_n}y_n = (\deg \; g)g. \]
Make the substitution $y_i \mapsto f_i$. 
Then we have 
\[h_1f_1 + h_2f_2 + \cdots +h_nf_n=0.\]
This shows the first assertion.  
By definition of $g=g(y)$, we have $g(f_1, f_2, \ldots, f_n)=0$. 
Apply the operator $\frac{\pa}{\pa x_j}$ to the this equality. Then we have 
\[0=\frac{\pa g(f_1, f_2, \ldots , f_n)}{\pa x_j}= 
\sum _{k=1}^n \frac{\pa g}{\pa y_k}(f_1, \ldots, f_n)\frac{\pa f_k}{\pa x_j} 
= \sum _{k=1}^nh_k\frac{\pa f_k}{\pa x_j}.  \]
This shows the second assertion. 
\end{proof}

\begin{theorem}  \label{duality}   
Let $f=f(x) \in K[x]$ be a homogeneous polynomial and put $f_j=\frac{\pa f}{\pa x_j}$.  
Assume that $f_1, \ldots , f_n$  are algebraically dependent. 
Let $\boldf=(f_1, \dots, f_n)$ and let 
$\boldh=(h_1, \ldots, h_n)$ be a system of forms as defined in 
Proposition~\ref{system_associated_to_alg_dep_system_is_svs} for $\boldf$.  
Then  
\begin{enumerate}
\item[{\rm (a)}] 
$f(x) \in \sol(\boldh; K[x])$.
\item[{\rm (b)}]
$f_j(x) \in \sol(\boldh;  K[x])$ for $j=1,2, \ldots, n$.
\item[{\rm (c)}]
$f(x) \in \sol(\pa \boldh/ \pa x_j; K[x])$ for $j=1,2, \ldots, n$. 
\item[{\rm (d)}]
$\boldh$ is a self-vanishing system of forms.
\end{enumerate}

\end{theorem}

\begin{proof}
The assertion (a) follows immediately from 
Proposition~\ref{system_associated_to_alg_dep_system_is_svs} (a).   
Proposition~\ref{system_associated_to_alg_dep_system_is_svs} (b) says that  
$\boldh$ is a syzygy of $\pa \boldf/ \pa x_j$.  This means that 
\[\sum _{k=1} ^{n} h_k \frac{\pa f_k}{\pa x_j}=0.\]
But $f_k=\frac{\pa f}{\pa x_k}$.  
Hence we have 
\[\sum _{k=1} ^{n} h_k \frac{\pa f_k}{\pa x_j}      = \sum _{k=1} ^{n} h_k \frac{\pa f_j}{\pa x_k}   =0.\]
This shows  assertion (b).  

Again by Proposition~\ref{system_associated_to_alg_dep_system_is_svs} (a) we have $\boldh \cdot \boldf=0$.  
For each $j$, we have   
\[  \frac{\pa}{\pa x_j}(\boldh\cdot \boldf)= \frac{\pa \boldh}{\pa x_j} \cdot \boldf + \boldh \cdot \frac{\pa \boldf}{\pa x_j}=0.\]
Again by Proposition~\ref{system_associated_to_alg_dep_system_is_svs} (b) we have $\boldh \cdot {\pa \boldf}/{\pa x_j}=0$.  
Hence ${\pa \boldh}/{\pa x_j} \cdot \boldf =0$.  This shows that $f(x) \in \sol(\frac{\pa}{\pa x_j}\boldh; K[x])$. 
Thus (c) is proved. 

Since $\sol(\boldh; K[x])$ is a commutative ring, we have 
\[K[f_1, \ldots, f_n]  \subset \sol(\boldh; K[x]). \]
Since $h_j(x)$ are polynomials in $f_1, \ldots, f_n$, 
this shows that $h_j \in \sol(\boldh; K[x])$. Thus  (d) is proved.  
\end{proof}

\begin{proposition}  \label{svs_behaves_as_variables}   
Let $K[x]=K[x_1, \ldots, x_n]$ be the polynomial ring and 
let  $\boldh$  be  a self-vanishing system of forms in $K[x]$.  
Then, for any $f(x) \in K[x]$, we have 
\[\cD _x(\boldh)\left(     f^{(i)}(x, \boldh)   \right) = (i+1)f^{(i+1)}(x, \boldh). \]
\end{proposition}
\begin{proof}
Let $y=(y_1, y_2, \ldots y_n)$ be a coordinate system independent of $x$. 
Note that 
\[
\frac{\pa f^{(i)}(x,\boldh)}{\pa x_j} =\left.\left\{\frac{\pa f^{(i)}(x,y)}{\pa x_j} \right\}\right|_{y \mapsto \tinyh} + 
\left.\left\{\sum _{k=1}^n  \frac{\pa f^{(i)}(x,y)}{\pa y_k}\frac{\pa h_k}{\pa x_j}   \right\}\right|_{y \mapsto \tinyh}, 
\]
so $\cD_x(\boldh)f^{(i)}(x, \boldh)$ is equal to
\begin{equation} \label{long_formula_2}
\sum _{j=1}^nh_j(x)\left.\left\{\frac{\pa f^{(i)}(x,y)}{\pa x_j} \right\}\right|_{y \mapsto \tinyh} 
+ \sum _{j=1}^n h_j(x)\left.\left\{\sum _{k=1}^n  \frac{\pa f^{(i)}(x,y)}{\pa y_k}\frac{\pa h_k(x)}{\pa x_j}   \right\}\right|_{y \mapsto \tinyh}, 
\end{equation}
where $y \mapsto \boldh$ means substitution.
Since $\cD_x(\boldh)h_k(x)=\sum _{j=1}^n h_j(x)\frac{\pa h_k(x)}{\pa x_j}=0$ for every $k=1,2,  \ldots, n$, 
the second summand of formula \eqref{long_formula_2} vanishes. The first summand of 
formula \eqref{long_formula_2} is equal to
$$\big\{\sum_{j=1}^n y_j \frac{\pa f^{(i)}(x,y)}{\pa x_j}\big\}\big|_{y \mapsto \tinyh}
= \big\{\cD_x(y)f^{(i)}(x,y) \big\} \big|_{y \mapsto \tinyh},$$ 
so
\[
\cD_x(\boldh)f^{(i)}(x, \boldh) = \left.\left\{\cD_x(y)f^{(i)}(x,y) \right\} \right|_{y \mapsto \tinyh}
= (i+1) f^{(i+1)}(x,\boldh)
\]
by definition of $f^{(j)}(x,y)$.
\end{proof}

\begin{theorem} \label{mainthm_1}   
Suppose that $\boldh=(h_1, \ldots, h_n)$ is a self-vanishing system of forms  in $K[x]$. 
Then, for a homogeneous polynomial $f(x) \in K[x]$, the following conditions are equivalent:
\begin{enumerate}
\item[$(a)$]
$f(x) \in \sol(\boldh; K[x])$.
\item[$(b)$]
$f^{(j)}(x, \boldh) =0, \ \mbox{ for } j=1,2, \ldots $
\item[$(c)$]
$f(x+t\boldh(x))=f(x) \ \mbox{for any } t \in K'$, where $K'$ is any extension field of $K$. 
\end{enumerate}
\end{theorem}

\begin{proof}
$(a) \Leftrightarrow (b)$ By Proposition~\ref{svs_behaves_as_variables}, it is enough to show the case for $j=1$.  
\begin{eqnarray*}
 f^{(1)}(x, \boldh) & =  & \left.\left\{\cD_x(y)f(x)\right\}\right|_{y \mapsto \tinyh} \\
                   & = &  \cD_x(\boldh)f(x)  \\ 
                   & = &  h_1\frac{\pa f}{\pa x_1}+ h_2\frac{\pa f}{\pa x_2}+ \cdots + h_n\frac{\pa f}{\pa x_n} =0. 
\end{eqnarray*}
The equivalence of $(b)$ and $(c)$ follows immediately from  Proposition~\ref{taylor_expansion}. 
\end{proof}

\begin{corollary}  \label{any_factor_is_in_sol}  
Let $\boldh$ be a self-vanishing system in $K[x]$.  If $f(x), g(x) \in K[x]\setminus \{0\}$, 
and if $f(x)g(x) \in \sol(\boldh; K[x])$, then both  $f(x), g(x) \in \sol(\boldh; K[x])$. 
\end{corollary}

\begin{proof}
Using Taylor expansion (Proposition~\ref{taylor_expansion}), we have 
\[f(x+t\boldh)g(x+t\boldh) = \sum _{j=0}^{\infty} \sum _{k+l=j} t^{k+l}f^{(k)}(x, \boldh) g^{(l)}(x, \boldh). \]
Let  $k_0$ be the highest degree for which $f^{(k_0)} \neq 0$, and similarly $l_0$ for 
$g^{(l_0)}(x, \boldh)$.  Then by Theorem~\ref{mainthm_1},  $f(x)g(x) \in \sol(\boldh; R)$ implies $k_0+l_0=0$.
Hence $k_0=l_0=0$.  Again by Theorem~\ref{mainthm_1}, proof is complete.
\end{proof}

\begin{corollary}  \label{solution_with_hn_inverted}  
Suppose that $\boldh=(h_1, \ldots, h_n)$ is a self-vanishing system in $R=K[x]$ such that $h_n(x) \neq 0$. 
Put \[s_i(x)=x_i - \frac{h_i(x)}{h_n(x)}x_n, \ (1 \leq i \leq n).\]
Let $f(x) \in \sol(\boldh; R)$.  Then 
\[f(x)= f(s_1, s_2, \ldots, s_{n-1}, 0),\]
and 
\[\sol(\boldh; K[x])= K[s_1, s_2, \ldots, s_{n-1}] \cap K[x].\]
\end{corollary}

\begin{proof}
It is easy to check by direct computation that 
\[\cD_x(\boldh)\left(\frac{1}{h_n(x)}\, \text{det}\!\begin{pmatrix}     x_i&x_n   \\    h_i(x)&h_n(x) \end{pmatrix}  \right)=0.\]
Hence $\cD_x(\boldh)s_j(x)=0$.  This shows that 
\[K[s_1, s_2, \ldots, s_{n-1}] \cap K[x] \subset \sol(\boldh; K[x]).\]
To show the converse, let $f(x) \in \sol(\boldh; K[x])$.  
Then $f(x) = f(x+ t \boldh)$ for an indeterminate $t$ by Theorem~\ref{mainthm_1}. 
Replace  $t$  for  $t= -\frac{x_n}{h_n(x)}$.  Then we get 
\[f(x)=f(s_1, s_2, \ldots, s_{n-1}, 0).\] 
\end{proof}

\begin{corollary}  \label{second_meaning_of_self_vanishing}   
Let $\boldh=(h_1, \ldots, h_n)$ be a self-vanishing system of polynomials. 
Let $f(x) \in \sol(\boldh; K[x])$.  
Suppose that $f(x)$ is homogeneous of positive degree $d$.  
Then we have 
\[f(h_1, h_2, \ldots, h_n)=0.\]
In particular, if each $h_j$ has positive degree, then  
\[h_j(h_1, h_2, \ldots, h_n)=0 \mbox{ for all } j=1,2, \ldots, n. \]
\end{corollary}

\begin{proof}
Let $y=(y_1, \ldots, y_n)$ be a coordinate system independent of $x$,  and put  
$\cD_x(y)^jf(x)=\frac{1}{j!}f^{(j)}(x,y)$.  We have shown that 
$f^{(d)}(x,y)=f(y)$.  In this equation substitute  $y$ for  $\boldh$.  Then we have 
$f^{(d)}(x,\boldh)=f(\boldh)$.   
In Theorem~\ref{mainthm_1}, we showed that 
$f^{(j)}(x, \boldh)=0$ for $j>0$. Thus we have $f(\boldh)=0$.        
\end{proof}

\section{Forms with zero Hessian and reduced self-vanishing systems} 

\begin{proposition}   \label{basic_1}    
Let $R=K[x_1, \ldots, x_n]$.  
Let 
$\boldf=(f_1, f_2, \ldots, f_n)$ be a system of forms in $R$.  
Then $\rank \left( \frac{\pa f_i}{\pa x_j} \right) = 
\trdeg_K\; K(f_1, f_2, \ldots, f_n)$.
In particular the following conditions are equivalent. 
\begin{enumerate}
\item[\upshape(1)]
$f_{i_1}, f_{i_2}, \ldots, f_{i_r}$ are algebraically dependent 
for every $i_1, i_2, \ldots, i_r \in \{1,2,\ldots,n\}$.
\item[\upshape(2)]
The rank of  Jacobian matrix $(\frac{\pa f_i}{\pa x_j})$ is $ < r$.
\item[\upshape(3)]
$\trdeg_K\; K(f_{i_1}, f_{i_2}, \ldots, f_{i_r}) < r$
for every $i_1, i_2, \ldots, i_r \in \{1,2,\ldots,n\}$.
\end{enumerate}
\end{proposition}

\begin{proof}
Left to the reader. 
\end{proof}

We say that a homogeneous polynomial $f \in R=K[x_1, \ldots, x_n]$ has zero Hessian if   
$\det\left(\frac{\pa ^2f}{\pa x_i \pa x_j}\right)=0$. 
(Whenever we discuss forms $f$ with zero Hessian, we assume $\deg \; f \geq 2$.)
Let $f$ be a homogeneous form in $K[x_1, \ldots, x_n]$.  
Since the Hessian determinant of $f$ is the Jacobian determinant of the partial derivatives of $f$, 
the following proposition follows immediately from Lemma~\ref{basic_1}. 

\begin{proposition} \label{basic_2}    
A homogeneous polynomial $f \in R$ is  a form with zero Hessian 
if and only if the partial derivatives
\[\frac{\pa f}{\pa x_1},\frac{\pa f}{\pa x_2},\ldots, \frac{\pa f}{\pa x_n} \]
are algebraically dependent. 
\end{proposition}

\begin{definition} \label{def_of_ideal_associated_to_form_with_zero_hessian}   
Let $y=(y_1, \ldots, y_n)$ be a coordinate system independent of $x$.  
Let  $f \in K[x]$ be a form with zero Hessian, and let 
\[\phi :K[y_1, \ldots, y_n] \rightarrow K[x_1, \ldots, x_n] \]
be the homomorphism defined by $\phi(y_j)= \frac{\pa f}{\pa x_j}$.  
We denote by $\cI(f)$ the kernel of $\phi$. $($By Proposition~\ref{basic_2}, $\cI(f) \neq 0$.$)$
\end{definition}

\begin{definition}  \label{reduced_system_of_polynomials}   
Let $f \in R=K[x]$ be a form with zero Hessian.  
Let $\cI(f) \subset K[y]$ be as in Definition~\ref{def_of_ideal_associated_to_form_with_zero_hessian}. 
Put $f_j=\frac{\pa f}{\pa x_j}$.  
Let $g(y)=g(y_1, \ldots, y_n) \in \cI(f)$ be 
a homogeneous form of the least degree in $\cI(f)$.  
Let \[h'_i(x_1, x_2, \ldots, x_n)=\frac{\pa g}{\pa y_i}(f_1, \ldots, f_n),\]
\[h_i(x_1, x_2, \ldots, x_n)=\frac{1}{{\rm GCD}(h'_1, h'_2, \ldots, h'_n)}h'_i(x_1, \ldots, x_n). \]
We call the vector $\boldh ':=(h'_1, h'_2, \ldots, h'_n)$ 
a system of polynomials arising from  $f(x)$, and 
$\boldh:=(h_1, \ldots, h_n)$ a reduced system of polynomials arising from $f(x)$.  
\end{definition}

\begin{remark}   
By Remark~\ref{non_vanishing_of_boldh}, 
$\boldh \neq 0$ as well as  $\boldh ' \neq 0$.   
\end{remark}

\begin{remark} \label{uniqueness_of_g_and_h}  
Assume that $\rank \left(\frac{\pa ^2 f}{\pa x_i \pa x_j}\right) =n-1$. Then 
the ideal $\cI(f)$ is a principal ideal of $K[y]$.    In this case 
$g \in \cI(f) \setminus \{0\}$ with the smallest degree is uniquely determined (up to a constant multiple). 
Hence $\boldh$, in Definition~\ref{reduced_system_of_polynomials}, is uniquely determined by $f$
(up to a nonzero element of $K$).  
On the other hand by Proposition~\ref{system_associated_to_alg_dep_system_is_svs} (b), we see that  
 $\boldh$ is a null vector of  the matrix  $\left(\frac{\pa ^2 f}{\pa x_i \pa x_j}\right)$.     
 Such a polynomial vector is unique up to a multiple of a polynomial.   
 Hence, for any system $\boldh$ of forms, we can show that $\boldh$ is the self-vanishing system as defined in
 Definition~\ref{reduced_system_of_polynomials} if and only if 
the following two conditions are satisfied. 
\begin{enumerate}
\item[\upshape(1)]
$(h_1, h_2, \ldots,  h_n) \left(\frac{\pa ^2 f}{\pa x_i \pa x_j}\right)  =0$,
\item[\upshape(2)]
$\text{GCD}(h_1, h_2 \ldots, h_n)=1$.
\end{enumerate}
\end{remark}

\begin{theorem}[Gordan-Noether]  \label{1st_main_th_of_gordan_noether}
Suppose that $f(x) \in K[x_1, \ldots, x_n]$ is a form with zero Hessian. 
Then a variable can be eliminated from $f$ and its partial derivatives 
simultaneously by means of a birational transformation of the variables.
\end{theorem}

\begin{proof}
Let $\boldh$ be a reduced self-vanishing system arising from  $f$.  
Then we have  $f(x) \in \sol(\boldh; K[x])$ 
and $\pa f(x)/\pa x_j \in \sol(\boldh; K[x])$ for $j = 1,2,\ldots,n$ 
by Theorem~\ref{duality}~(a) and (b) respectively.
By Lemma~\ref{basic_001} and Remark~\ref{non_vanishing_of_boldh}, we may assume that $h_n \neq 0$.  
Put $s_j=x_j - \frac{h_j}{h_n}x_n$, for $j=1, \ldots, n-1$. Set $s_n=0$.   Then by    
Corollary~\ref{solution_with_hn_inverted},  $f$ is a polynomial in $s_1, \ldots, s_{n-1}$,
and so are $\pa f(x)/\pa x_j$ for $j = 1,2,\ldots,n$.
  We claim that 
\[K(s_1,\ldots, s_{n-1}, x_n)=K(x_1, \ldots, x_n).\]
In fact we have  
\[x_j=s_j + \frac{h_j(x)}{h_n(x)}x_n, \ \ j=1,2, \ldots, n-1.\] 
Since $h_{j}(x) \in \sol(\boldh;K[x])$,  we have $h_j(x+t\boldh(x))=h_j(x)$ for any 
$t$ in any extension field of $K$ by Theorem~\ref{mainthm_1} (c).    
Now let  $t=-\frac{x_n}{h_n}$. Then $h_j(s)=h_j(x)$. 
This shows that $x_j\in K(s_1, \ldots, s_{n-1},x_n)$ for all $j$, as desired.  
\end{proof}

\begin{proposition} \label{basic_3}   
Suppose that 
$f(x) \in K[x]$ is a form with zero Hessian. 
Let ${\cal I}(f(x))$ be the ideal of 
$K[y]$ as  defined in Definition~\ref{def_of_ideal_associated_to_form_with_zero_hessian}.  
Then the following conditions are equivalent. 
\begin{enumerate}
\item[{\rm (a)}]
The ideal ${\cal I}(f(x))$ contains a linear form. 
\item[${\rm (b)}$]
The partial derivatives of $f(x)$ are linearly dependent. 
\item[{\rm (c)}]
A variable can be eliminated from  $f(x)$ by means of a linear transformation of the variables. 
\end{enumerate}
\end{proposition}    

\begin{proof}
The equivalence of (a) and (b) is clear.  Suppose that there exists a non-trivial relation 
\[a_1 f_1 + a_2 f_2 + \cdots + a_n f_n =0,\]
where $f_j= \frac{\pa f}{\pa x_j}$ and $a_j \in K$.   
It is possible to choose a set of linearly  independent linear forms  
$y_1, \ldots, y_n$  in $x_1, \ldots, x_n$ such that 
$\frac{\pa x_j}{\pa y_1}=a_j$.  
Then 
\[\frac{\pa f}{\pa y_1}= \frac{\pa f}{\pa x_1}\frac{\pa x_1}{\pa y_1}+ \frac{\pa f}{\pa x_2}\frac{\pa x_2}{\pa y_1}+ 
\cdots + \frac{\pa f}{\pa x_n}\frac{\pa x_n}{\pa y_1}=0.\] 
This shows that if $f$ is expressed in terms of $y_j$, then $f$ does not contain $y_1$.  Thus (b) $\Rightarrow$ (c). 
The same argument shows (b) $\Leftarrow$ (c) as well.   
\end{proof}

\begin{theorem}  \label{basic_fact_due_to_gordan_noether}
Let $f(x) \in K[x]$ be  a form with zero Hessian and let  
$\boldh'=(h'_1, \ldots, h'_n)$ be a system of 
polynomials associated to an element $g=g(y_1, \ldots, y_n) \in \cI(f(x))$ of the least degree.  
Similarly let $\boldh=(h_1, \ldots, h_n)$ be the reduced system of polynomials defined by 
\[\boldh= \frac{1}{{\rm  GCD}(\boldh ')}\boldh'.\]
(See Definition~\ref{reduced_system_of_polynomials}.)
Then we have:
\begin{enumerate}
\item[$(${\rm a}$)$]
$\boldh$ and $\boldh '$ are self-vanishing systems of polynomials.
\item[$(${\rm b}$)$]
$f(x) \in \sol(\boldh; K[x])$. 
\item[$(${\rm c}$)$]
$\frac{\pa }{\pa x_j} f(x) \in \sol(\boldh; K[x])$, for all $j=1,2, \ldots, n$.
\end{enumerate}
\end{theorem}

\begin{proof}
(a) 
By Theorem~\ref{duality}, $\boldh'(x)$ is a self-vanishing system. 
Note that $\sol(\boldh '; K[x])=\sol(\boldh; K[x])$. 
By Corollary~\ref{any_factor_is_in_sol}, 
$\boldh(x)$ is a self-vanishing system. 
(b) and (c) are proved in Theorem~\ref{duality}.
\end{proof} 

\begin{example} \label{no_splitting_case}
If $f = (x_1^2 x_3 + 2 x_1 x_2 x_4 + x_2^2 x_5) 
(x_1^2 x_4 + 2 x_1 x_2 x_5 + x_2^2 x_6)$, then
$\cI(f)$ is principal, and
$$
\boldh = (0, 0, x_2^3, -x_2^2 x_1, x_1^2 x_2, -x_1^3),
$$
is the unique reduced self-vanishing system arising from $f$, 
which is as in Example \ref{splitting_case}. 

But if $f = (x_3 x_1 + x_4 x_2) (x_5 x_1 + x_6 x_2)$, then
$\cI(f)$ is principal as well, and
$$
\boldh =
\big(0, 0, (x_1 x_3 + x_2 x_4) x_2, -(x_1 x_3 + x_2 x_4) x_1, 
-x_2 (x_1 x_5 + x_2 x_6), x_1 (x_1 x_5 + x_2 x_6)\big),
$$
is the unique reduced self-vanishing system arising from $f$, 
which is \emph{not} as in Example \ref{splitting_case}.
\end{example}

\section{Binary and ternary forms with zero Hessian}

\begin{theorem}   
Assume that $n=2$ and let $f \in K[x_1, x_2]$ be a form of degree $d$ 
with zero Hessian. 
Then $f=(a_1x_1+a_2x_2)^d$ for some $a_1, a_2 \in K$.  
\end{theorem}

\begin{proof} Let ${\cal I}(f) \subset K[y_1, y_2]$ be the ideal defined in 
Definition~\ref{def_of_ideal_associated_to_form_with_zero_hessian}.
Since  ${\cal I}(f)$ is a prime ideal, it is a principal ideal 
generated by a linear form. 
Hence the assertion follows from Proposition~\ref{basic_3}. 
\end{proof}

\begin{lemma} \label{rank_of_hessian_matrix_corrected}
Let $\bm{h}$ be a homogeneous self-vanishing system in dimension $n$. Then
$$
\rk \Big(\parder{h_i}{x_j}\Big) = \trdeg_K K(\bm{h}) = \krdim K[\bm{h}] \le n-1
$$
and if $n \ge 3$ in addition, then
$$
\rk \Big(\parder{h_i}{x_j}\Big) = \trdeg_K K(\bm{h}) = \krdim K[\bm{h}] \le n-2
$$
\end{lemma}

\begin{proof}
The case where $\bm{h} = 0$ is trivial, so let us assume without loss of 
generality that $h_1 \ne 0$. From Theorem 3.11 (c), it follows that
$h_1(x+t\bm{h}(x)) = h_1(x)$. If we look at the leading coefficient with respect to
$t$, we see that $h_1(\bm{h}) = 0$. So
$$
\rk \Big(\frac{\partial h_i}{\partial x_j}\Big) = \trdeg_K K(\bm{h}) \le n-1. 
$$
This gives the case $n \le 2$, so assume from now on that $n \ge 3$.

Let $f$ be an irreducible factor of $h_1$. If each $h_i$ is a $K$-multiple 
of a power of $f$, then the components of $\bm{h}$ are linearly dependent
in pairs, and
$$
\rk \Big(\frac{\partial h_i}{\partial x_j}\Big) = \trdeg_K K(\bm{h}) \le 1 \le n - 2. 
$$
Otherwise, there exists an $h_i$ with an irreducible factor $f'$ which is 
not a $K$-multiple of $f$.
By Corollary~\ref{any_factor_is_in_sol},
we have $f'(x) \in \sol(\boldh; K[x])$ as well as  $f(x) \in \sol(\boldh; K[x])$.
From Theorem~\ref{mainthm_1}  it follows that   
$f(x+t\bm{h}(x)) = f'(x+t\bm{h}(x)) = 0$.  If we look at the leading coefficient 
with respect to $t$, we see that 
\begin{equation} \label{fh0} 
f(\bm{h}(x)) = f'(\bm{h}(x)) = 0,  
\end{equation}
which gives the second claim of Lemma~\ref{rank_of_hessian_matrix_corrected}.
\end{proof}

\begin{theorem} \label{ternary_form_with_zero_hessian_is_trivial}  
Suppose that $f=f(x) \in K[x_1, x_2, x_3]$ is a form with zero Hessian.  Then 
a variable can be eliminated from $f$ by means of a linear transformation of variables. 
\end{theorem}

\begin{proof}
Let $\boldh=(h_1, h_2, h_3)$  be a reduced system of polynomials arising from $f(x)$ 
(See Definition~\ref{reduced_system_of_polynomials}).  We claim that 
$\boldh$ is  a constant vector. Suppose that it is not. 
Then ${\rm Krull\,dim}\;K[h_1, h_2, h_3]=1$ by Proposition~\ref{rank_of_hessian_matrix_corrected}. 
Thus  any two of the elements $h_1, h_2, h_3$ has a homogeneous algebraic relation. 
Since $K[h_1, h_2, h_3]$ 
is an integral domain, and $K$ is algebraically closed,  they  should be  linear relations. Thus we have 
$\dim _K\; Kh_1 + Kh_2 + Kh_3=1$.  Since ${\rm GCD}(h_1, h_2, h_3)=1$, this is impossible unless they are  constants. 
Recall that we have 
\[h_1\frac{\pa f}{\pa x_1}+ h_2\frac{\pa f}{\pa x_2}+ h_3\frac{\pa f}{\pa x_3}=0.\]
By Proposition~\ref{basic_3}, proof is complete. 
\end{proof}

\section{The rational map defined by $\boldh$ }
In this section, we assume that $n \geq 4$.  
Let $R=K[x_1, \ldots, x_n]$ 
and let $\boldh = (h_1, \ldots, h_n)$ be any homogeneous reduced self-vanishing system.
So ${\rm GCD}(h_1, \ldots, h_n) = 1$.

Let \[Z : \PP ^{n-1}(x) \rightarrow \PP^{n-1}(y)\]  
be the rational map defined by the correspondence $x=(x_1: \cdots : x_n) \mapsto (h_1 : \cdots : h_n)$.  
Let $W$ be the image of $Z$  and  $T$ the fundamental locus of $Z$ in $\PP^{n-1}(x)$ defined by 
the equations  $h_1(x)=h_2(x)= \cdots =h_n(x)=0$.
The algebraic set $W \subset \PP^{n-1}(y)$ is defined by the kernel of the homomorphism
defined by 
\[y_j \mapsto  h_j, (j=1, 2, \ldots, n).\] 
which corresponds to $Z$. 

\begin{proposition} \label{yamada's_corollary_2_page38}   
The following conditions are equivalent.
\begin{enumerate}
\item[{\rm (a)}]
$\deg \; (h_j)=0$, i.e., $\boldh$ is a constant vector.

\item[{\rm (b)}]
$\dim \; W =0$, i.e., $W$ is a one-point set. 

\item[{\rm (c)}]
$T$ is empty, i.e., $Z$ is a morphism. 
\end{enumerate}
\end{proposition}

\begin{proof}
Suppose that $\boldh$  is not a constant.  Then we have  
\[h_j(h_1, \ldots, h_n)=0 \ \text{for every} \ j=1, \ldots, n.\] 
by Corollary~\ref{second_meaning_of_self_vanishing}.  
Thus any specialization of $(h_1, \ldots, h_n)$ is a point of $T$. 
This shows that if  $T$ is empty, then $\boldh$ is a constant vector.  All other implications are trivial.  
\end{proof}


\begin{proposition}  \label{dimension_of_T} \label{yamada's_proposition_3_page38}  
If $\dim \; W \geq 1$, then 
\[ 2 \leq  \mbox{\rm Krull\,dim}\; K[x]/(h_1, \ldots, h_n) \leq n-2,\]
or equivalently,  
\[ 1 \leq \dim \; T   \leq n-3.\]
\end{proposition}

\begin{proof}
Since $T$ is not empty, the ideal $(h_1, h_2, \ldots, h_n) \subset K[x]$ is not the unit ideal by the previous proposition. 
On the other hand, since ${\rm GCD}(h_1, \ldots, h_n) = 1$, it is not a principal ideal. 
Hence ${\rm ht}(h_1, \ldots, h_n)  \geq 2$.  This shows that $\dim  \; T = {\rm Krull\,dim}\; K[x]/(h_1, \ldots, h_n) -1 \leq n-3$.  
On the other hand there exists a surjective homomorphism of rings:
\[ K[x_1, \ldots, x_n]/(h_1, \ldots, h_n) \rightarrow K[h_1, \ldots, h_n], \]
\[   x_j \mapsto h_j, \]
provided that $\deg \; \boldh > 0$. 
Note that $K[x_1, \ldots, x_n]/(h_1, \ldots, h_n)$ is the fiber at the origin of the
inclusion map  $K[h_1, \ldots, h_n] \rightarrow K[x_1, \ldots, x_n]$ 
and ${\rm Krull\,dim} \; K[h_1, \ldots, h_n] \leq n-2$ by Lemma~\ref{rank_of_hessian_matrix_corrected}. 
Hence  we have 
${\rm Krull\,dim}\;  K[x_1, \ldots, x_n]/(h_1, \ldots, h_n) \geq 2$.
\end{proof}

For the rest of  this section we assume that $\dim \; W=1$.  In this case the fiber 
of $Z : \PP^{n-1}(x) \rightarrow \PP ^{n-1}(y)$ is a hypersurface 
of $\PP ^{n-1}(x)$.  Thus, for any $\om \in \PP^{n-1}(y)$, $Z^{-1}(\om)$ is defined by one homogeneous polynomial in $x=(x_1, \ldots,  x_n)$.  

\begin{definition}  
Let $Z : \PP ^{n-1} (x) \rightarrow \PP ^{n-1} (y)$ be as above. 
Let $\om =(\om _1 : \om _2: \cdots : \om _n) \in \PP^{n-1}(y)$.  
We denote by $g^{(\om)}(x)$ the square-free polynomial 
in $K[x]$ that defines the hypersurface of the fiber of $Z$ at $\om \in W \subset  \PP ^{n-1}(y)$. 
\end{definition}

For each point $\om=(\om_1:\om_2: \cdots : \om_n) \in \PP ^{n-1}$, 
we define the differential operator $\cD_x(\om)$  on $R=K[x_1, \ldots, x_n]$ by 
\[\cD_x(\om)f(x)= \sum _{j=1}^n\ \om_j \frac{\pa f(x)}{\pa x_j} \]
whose value is determined up to a non-zero constant factor.  Hence 
we may speak of the set of solutions  of the equation $\cD_x(\om)f(x)=0$.    
We denote the space of solutions  of $\cD_x(\om)f(x)=0$ in $R=K[x_1, \dots, x_n]$ 
by $\sol(\om;R)$.     
It is a subring of $R$.  
For any subset $U$ of $\PP ^{n-1}$ we denote by $\sol(U; R)$ the space of solutions 
 of the system of linear differential equations 
\[\cD_x(\om)(f(x))=0, \ \om  \in U.\]
If we denote by $L(U)$ the linear closure of  $U$   in $\PP ^{n-1}$ it is easy to see that  
\[\sol(U; R)=\sol(L(U); R).\]
If $U=\{\om ^{(1)}, \om ^{(2)}, \ldots , \om ^{(s)}\}$  is a finite set,  then  
$\sol(U;R)$ is also denoted as 
\[\sol(\om ^{(1)}, \om ^{(2)}, \ldots , \om ^{(s)}; R).\]
The same notation is used if we replace $\PP ^{n-1}(y)$ for the vector space $K^n$ in the obvious sense.  
Namely for a linear subspace $L$ of $K^n$, we denote by 
$\sol(L; R)$, the set of solutions of 
the differential equations 
\[ \cD_x(\bolda)F(x)  =0 {\mbox{ for all } \bolda \in L},\]
where $\bolda$ denotes a row vector in $L$ regarded as a system of constants.    
In a  set theoretic notation, 
\[\sol(L; R)=\left\{F(x) \in R \left| \left(a_1 \frac{\pa}{\pa x_1}+ a_2 \frac{\pa}{\pa x_2} + \cdots + a_n \frac{\pa}{\pa x_n}\right)F(x)=0, 
(a_1, \ldots, a_n) \in L   \right. \right\}. \]
Note that $\sol(L;R)$ is a subring of $R$ generated by homogeneous linear forms.
 
The relation between the subring $\sol(L; R)$ of $R$ and the ideal $I$ for the linear subspace in $\PP ^{n-1}$  is  very important for us.  
In the next theorem and corollary we describe a set of generators of the subring $\sol(L; R)$ and a set of generators that defines 
the linear space $L$ as a subspace of $\PP ^{n-1}$. 


\begin{theorem}  \label{subring_by_generated_by_determinates}   
Let $R=K[x_1, \ldots, x_n]$ be the polynomial ring over $K$. 
Let $A=(a^{(j)}_i)$ be a $k \times n$ matrix, $a^{(j)}_i \in K$,  with 
$\bolda^{(j)}=(a^{(j)}_1, a^{(j)}_2, \ldots, a^{(j)}_n)$ as the $j$-th row. 
Suppose that the  rows are linearly independent.  
Assume that $k <  n$.  Then the set of solutions as a subring of $K[x]$ 
\[\sol(\bolda^{(1)}, \ldots, \bolda^{(k)}; R):=\bigcap _{j=1}^k \sol(\bolda^{(j)};R)\] 
is isomorphic to the polynomial ring in $n-k$ variables.  It is generated by 
the $(k+1) \times (k+1)$ minors of the matrix 
\[
A':=\left(\begin{array}{ccccc}
a^{(1)}_1  &  a^{(1)}_2 & \cdots  & a^{(1)}_{n-1} & a^{(1)}_n \\ 
a^{(2)}_1  &  a^{(2)}_2 & \cdots  & a^{(2)}_{n-1} & a^{(2)}_n \\ 
          &   &   &   &                                   \\
a^{(k)}_1  &  a^{(k)}_2 & \cdots  & a^{(k)}_{n-1} & a^{(k)}_n \\  
\hline 
x_1  & x_2   &  \cdots & x_{n-1}& x_n
\end{array}
\right)
\]  
\end{theorem}

\begin{proof}
Let $V$ be the vector space of common syzygies of $\bolda^{(1)}, \ldots, \bolda^{(k)}$
over $K$. Take $\boldv = (v_1,v_2,\ldots,v_n) \in V$ and define a linear form
\begin{equation} \label{Qdef}
l_{\boldv} := v_1 x_1 + v_2 x_2 + \cdots + v_n x_n
\end{equation}
Let $k = n - \dim \, V$. Take a basis 
$\boldv_1, \boldv_2, \allowbreak \ldots, \boldv_{n-k}$ of $V$, 
and extend it to a basis $\boldv_1, \boldv_2, \ldots, \boldv_n$ of $K^n$. 
Let $l_1, l_2, \ldots, \allowbreak l_n$ be the corresponding linear forms as 
defined in \eqref{Qdef}. Then 
$$
R = K[l_1,l_2,\ldots,l_n].
$$
Now take $f \in R$. Then we can write
$f = g(l_1,l_2,\ldots,l_n)$ where $g \in K[y_1,y_2,\ldots,y_n]$.
Write $f_i = \pa f / \pa x_i$ and $g_i = \pa g / \pa y_i$ for $i = 1,2,\ldots,n$. 
Then
$$
(f_1,f_2,\ldots,f_n) = g_1(l_1,l_2,\ldots,l_n) \boldv_1 + g_2(l_1,l_2,\ldots,l_n) \boldv_2 + 
\cdots + g_n(l_1,l_2,\ldots,l_n) \boldv_n
$$
Hence $f \in \sol(\bolda^{(1)}, \ldots, \bolda^{(k)}; R)$
if and only if $g_{n-k+1} = g_{n-k+2} = \cdots = g_n = 0$, i.e., 
$f \in K[l_1,l_2,\ldots,l_{n-k}]$. Indeed, $K[l_1,l_2,\ldots,l_{n-k}]$
in isomorphic to the polynomial ring in $n - k$ variables.

If we take for $f$ a $(k+1) \times (k+1)$ minor of $A'$, then 
$f(\bolda^{(j)}) = 0$ because $f(\bolda^{(j)})$ is the determinant of a matrix 
of which the last row coincides with row $j$. From this, we infer that
the $(k+1) \times (k+1)$ minor of $A'$ are linear combinations
of $l_1,l_2,\ldots,l_{n-k}$. So it remains to show the converse, i.e.,
that $l_{\boldv}$ is a linear combination of the $(k+1) \times (k+1)$ minors of 
$A'$ for every $\boldv \in V$.

Suppose first that $k = n - 1$. Then $V = K w_1$ has dimension $1$, and
$l_{\boldv}$, $l_1$ and $\det A'$ are the same up to a nonzero constant
for every nonzero $\boldv \in V$. So $l_{\boldv}$ is a linear combination 
of the $(k+1) \times (k+1)$ minors of $A'$ for every $\boldv \in V$.

Suppose next that $k < n - 1$. Then we can extend 
$\bolda^{(1)}, \ldots, \bolda^{(k)}$ to a basis $\bolda^{(1)}, \ldots, \bolda^{(n-1)}$
of the syzygies over $K$ of $\boldv$. Now the case $k = n - 1$ yields 
$l_{\boldv}$ as a determinant of an $n \times n$ matrix up to a nonzero 
constant. If we expand this matrix along rows $k+1, k+2, \ldots,n-1$, 
then we get a linear combination of the $(k+1) \times (k+1)$ minors of $A'$. 
\end{proof}

Corollaries~\ref{ideal_generators_for_L} and \ref{second_description_of_linear_forms} that follow 
are other ways to describe the set of linear forms as generators for $\sol(L;R)$. 

\begin{corollary}   \label{ideal_generators_for_L} 
Let $R$ and $A$ be the same  as Theorem~\ref{subring_by_generated_by_determinates}, and furthermore 
let $L$ be the vector subspace in $K^n$ generated by the rows of $A$. 
Then the set of solutions $\sol(L; R)$, as a subring of $R=K[x_1, \ldots, x_n]$,  is  generated by 
the linear forms $l=l(x)$ such that 
\[l(a^{(j)}_1, a^{(j)}_2, \ldots, a^{(j)}_n)=0, \] 
for all $j=1,2, \dots, k$. 
\end{corollary}
\begin{proof}
This follows from the proof of Theorem \ref{subring_by_generated_by_determinates}.
\end{proof}

\begin{corollary}  \label{second_description_of_linear_forms} 
Suppose that $U \subset \PP ^{n-1}(y)$ is a subset and let $L(U)$ be  the linear closure of $U$. 
Then $\sol(U; R)=\sol(L(U);R)$ is generated by the linear forms of  
\[\{\  l(x) \in K[x_1, x_2, \ldots, x_n] | l(\om _1, \om _2, \ldots, \om _n)=0 \ \mbox{ for all } (\om_1 : \om _2: \cdots : \om _n ) \in L(U) \},\]
as a subring of $K[x]$.    
\end{corollary}

The following two theorems are very important to determine the forms in four and five variables with zero Hessian. 

\begin{theorem}  \label{1st_main_thm_in_the_case_dim_W=1}   
Let $\boldh = (h_1, \ldots, h_n)$ be a self-vanishing system of forms in $R=K[x_1, \ldots, x_n]$.  
Let $Z:\PP ^{n-1}(x) \rightarrow \PP ^{n-1}(y)$ be the rational map defined by  $y_j =  h_j(x)$. 
Let $T \subset \PP ^{n-1}(x)$ be the fundamental locus and  $W$  the image of $Z$.  
Assume that $\dim \;W=1$. 
Let 
\[i: \PP ^{n-1}(y) \rightarrow \PP ^{n-1}(x)\] be the natural map $y_j \rightarrow x_j$.   
Then $i(L(W)) \subset T$, where  $L(W)$ is the linear closure of $W$ in $\PP ^{n-1}(y)$. 
\end{theorem}
\begin{proof}
Recall that $T$ is defined by the polynomials $h_1(x), \ldots, h_n(x)$. 
On the other hand $i(L(W))$ is defined by linear forms which vanish on the set $i(L(W))$. Hence,  
in view of Corollary~\ref{second_description_of_linear_forms}, this follows immediately 
from Theorem~\ref{2nd_main_thm_in_the_case_dim_W=1} below.    
\end{proof}

\begin{theorem} \label{2nd_main_thm_in_the_case_dim_W=1}   
With the same notations and assumptions as in Theorem~\ref{1st_main_thm_in_the_case_dim_W=1}, 
we have
\[h_j(x) \in \sol(L(W);R) \  \mbox{for  all}\ j=1,2, \ldots, n. \] 
\end{theorem}

We prove it after some propositions.  For the rest of this section we 
fix notation and assumption of Theorem~\ref{1st_main_thm_in_the_case_dim_W=1}.
In particular it is  assumed  that $\dim \;W=1$.
 
\begin{proposition} \label{main_proposition_1}  
For any $\om \in W$, we have $g^{(\om)}(x) \in \sol(\boldh; R)$. 
\end{proposition}
\begin{proof}
Choose a hyperplane $H \subset \PP ^{n-1}(y)$ such that $\om \in H$, and suppose that $H$ is defined by the linear equation 
\[a_1y_1 + \cdots + a_ny_n=0.\]  
Put $f(x)=\sum _{j=1}^na_jh_j(x)$.  Since $K$ is infinite, it is possible to choose $H$ such that $f(x) \neq 0$.  
Then for any $\al \in Z^{-1}(\omega)$, if $\al \in T$, then $h_j(\al)=0$ for all $j$ by definition  of $T$. 
Hence $f(\al)=0$.  If $\al \not \in T$, then 
\[f(\al)=\sum _{j=1}^n a_jh_j(\al)= c\sum _{j=1}^n a_j\om _j=0.\]
This shows that any point on $Z^{-1}(\om)$ is a zero of $f(x)$.  Hence $f(x)$  is a multiple of  $g^{(\om)}(x)$. 
By Corollary~\ref{any_factor_is_in_sol}, this proves that $g^{(\om)}(x) \in \sol(\boldh; R)$ for any $\om \in W$. 
\end{proof}

\begin{proposition} \label{main_proposition_2}   
For each point $\om \in W$, we have $g^{(\om)}(x) \in \sol(L(W); R)$.  
\end{proposition}

Proof is preceded by two lemmas. 

\begin{lemma}   \label{technical_lemma_1}   
For $\omega, \omega ' \in W$, the polynomial     
$\cD_x(\omega ')g^{(\omega)}(x)$ is divisible by  $g^{(\omega ')}(x)$. 
\end{lemma}

\begin{proof}
In Proposition~\ref{main_proposition_1}, we proved that $g^{(\omega)}(x) \in \sol(\boldh; R)$.
So if $\cD_x(\boldh)$ is applied to $g^{(\omega)}(x)$,  it becomes $0$.  
Namely, 
\[h_1(x)\frac{\pa g^{(\omega)}(x)}{\pa x_1}+ \cdots + h _n(x)\frac{\pa g^{(\omega)}(x)}{\pa x_n}=0. \] 
Choose $\al \in Z^{-1}(\om ')$ and make the substitution  $x \mapsto \al$ in the above equality.
Then the  left hand side  is the same as  
\[\cD_x(\om ')g^{(\om)}(x)\] evaluated at $x=\al$. 
We have shown that the zero locus of 
$g^{(\om ')}(x)$ is contained in that of $\cD_x(\om ')g^{(\om)}(x)$.  
Since $g^{(\om ')}(x)$ is square-free,  this proves the assertion.
\end{proof}

\begin{lemma}   \label{technical_lemma_2}   
For any point $\omega \in W$, we have $\cD_x(\omega)g^{(\omega)}(x)=0$. 
\end{lemma}

\begin{proof}
In the previous lemma let $\om ' = \om$.  Then we have proved that $\cD_x(\omega)g^{(\omega)}(x)=r(x) g^{(\omega)}(x)$ 
for some $r(x) \in R$. For the degree reason, we get the assertion.
\end{proof}

\begin{proof}[Proof of Proposition~\ref{main_proposition_2}.] 
By Lemma~\ref{technical_lemma_2}, we have 
\[\cD_x(\omega) \left(\cD_x(\omega ')g^{(\omega)}(x)\right)
=\cD_x(\omega ')\left( \cD_x(\omega)g^{(\omega)}(x)  \right)=0.\]
By Lemma~\ref{technical_lemma_1}, we may write  $\cD_x(\omega ')g^{(\omega)}(x)=r(x)g^{(\om ')}(x)$
for some $r(x)$.  Since 
\[\cD_x(\omega)\left( r(x)g^{(\omega ')} (x) \right) =0,    \]
$r(x)g^{(\omega ')}(x) \in \sol(\omega; R)$.  Therefore, since a constant vector is a 
self-vanishing system, we have $g^{(\omega ')}(x) \in \sol(\omega; R)$ by  Corollary~\ref{any_factor_is_in_sol}. 
\end{proof}

\begin{proof}[Proof of Theorem~\ref{2nd_main_thm_in_the_case_dim_W=1}.]
We have to show that $h_j(x) \in \sol(L(W); R)$.  
We may assume that $h_j \neq 0$. 
Let $W_j$ be the intersection of $W$ and the hyperplane defined by $y_j=0$ in the space $\PP ^{n-1}(y)$, 
and let  $H_j \subset \PP ^{n-1}(x)$ be the hypersurface defined by $h_j(x)=0$.  
Then $W_j$ is a finite set, since $\dim \; W=1$ and $W$ is not contained in the 
above hyperplane.
Thus,  if we write   
$W_j=\{\omega ^{(1)}, \omega ^{(2)}, \ldots, \omega ^{(s)} \}$, then   
\[H_j \supset Z^{-1}(\om ^{(1)}) \cup Z^{-1}(\om ^{(2)}) \cup \cdots \cup Z^{-1}(\om ^{(s)}).\]

Conversely if $\al \in H_j \setminus T$, then $Z(\al) \in W_j$. 
Thus 
\[H_j \setminus T \subset Z^{-1}(\omega ^{(1)}) \cup Z^{-1}(\omega ^{(2)}) \cup \cdots \cup Z^{-1}(\omega ^{(s)}).  \] 
Since $H_j$ is purely of codimension $1$ and $T$ has codimension 
at least 2 (Proposition \ref{dimension_of_T}), it in fact shows that  
\[H_j  \subset  Z^{-1}(\omega ^{(1)}) \cup Z^{-1}(\omega ^{(2)}) \cup \cdots \cup Z^{-1}(\omega ^{(s)}).  \] 
Thus we have proved that up to a nonzero constant factor, the square-free part of 
$h_j(x)$ is equal to
\[
g^{(\omega ^{(1)})}(x) \times g^{(\omega ^{(2)} )}(x) \times \cdots \times g^{(\omega ^{(s)})}(x).\] 
By Proposition~\ref{main_proposition_2}, this completes the proof of Theorem~\ref{2nd_main_thm_in_the_case_dim_W=1}.  
\end{proof}

\section{Quaternary and quinary forms with zero Hessian }  

\begin{theorem}  \label{quaternary_form_with_zero_hessian}   
Let $f=f(x_1, x_2, x_3, x_4) \in R =  K[x_1, x_2, x_3, x_4]$ be  a form with zero Hessian. 
Then $f$ can be transformed into a form with three variables via a linear transformation of the variables $x_1, x_2, x_3, x_4$. 
\end{theorem}

\begin{proof}
Let $\boldh=(h_1, h_2, h_3, h_4)$ be a reduced self-vanishing system of forms arising from $f$.  
Then by Lemma~\ref{rank_of_hessian_matrix_corrected}, 
${\rm Krull\,dim} \;K[h_1, h_2, h_3, h_4] \leq n-2=2$. 
This shows $\dim \; W \leq 1$.  
Put $s=\dim \; L(W)$ (as a linear variety in $\PP ^{3}$).  Note that $s+1=\dim _K (Kh_1 + \cdots + Kh_4)$.  
Since $i(L(W)) \subset T$ by Theorem~\ref{1st_main_thm_in_the_case_dim_W=1} and $\dim \; T \leq 1$ 
by Proposition~\ref{yamada's_proposition_3_page38}, we have $s \leq 1$. 
If $s=0$, we get the result by Proposition~\ref{yamada's_corollary_2_page38}. 
If $s=1$, we may assume that $h_1(x)=h_2(x)=0$ by Lemmas~\ref{basic_001} and \ref{num_of_var_involved_1}.
Let $g(y) \in K[y_1, \ldots, y_4]$ be the polynomial through which $h_1, \ldots, h_4$ 
are defined (cf.\@ Definition~\ref{reduced_system_of_polynomials}). 
Then $h_1=h_2=0$ implies that $\frac{\pa g}{\pa y_1}=\frac{\pa g}{\pa y_2}=0$ by the minimality of the degree.   
(cf.\@ Proposition~\ref{num_of_var_involved_2} and Remark~\ref{non_vanishing_of_boldh}.)  
Hence  $g$ is a form only in two variables.  Since $g$ should be irreducible, $g$ is a linear form.  
This implies that there exists a linear relation among the partial derivatives of $f$.  Hence the 
assertion follows from Proposition~\ref{basic_3}.
\end{proof}

\begin{proposition}   \label{svs_for_forms_with_zero_hessian_in_5_variables}  
Let $f=f(x_1, x_2, x_3, x_4, x_5) \in R = K[x_1, x_2, x_3, x_4, x_5]$ be  a form with zero Hessian. 
Let $\boldh=(h_1, h_2, h_3, h_4, h_5)$ be a self-vanishing system of polynomials associated to $f$ 
as defined in Definition~\ref{reduced_system_of_polynomials}.   
Assume that there exists no linear relations among the partial derivatives of $f$.  
Then by a linear change of variables,  $f$ can be transformed into a form so that  
$h_1=h_2=0$, and $h_3, h_4, h_5$ are polynomials only in $x_1, x_2$.
\end{proposition}

\begin{proof}
Let $Z:\PP ^{4}(x) \rightarrow \PP ^4(y)$ be the rational map defined by $y_j = h_j(x)$ and 
let $T$ be the fundamental locus and  $W$  the image of $Z$. 
\emph{We will prove that $\dim \; W \leq 1$ in the next section.}
If  $\dim \; W=0$, then
a variable can be eliminated from $f$, since $\boldh$ is a constant vector by 
Proposition~\ref{yamada's_corollary_2_page38}. 
Assume that $\dim\; W=1$.  
Then we have  $i(L(W)) \subset T$ by Theorem~\ref{1st_main_thm_in_the_case_dim_W=1}. 
On the other hand we have  $\dim \;T \leq 2$ by Proposition~\ref{yamada's_proposition_3_page38}. 
This shows $\dim \; L(W) \leq 2$ as a linear subspace in $\PP^{4}$.  
Let $s=\dim \; L(W)$ or equivalently  $s+1=\dim _K\sum _{j=1}^5Kh_j$.  
Since $\dim \; W =1$, $s \geq 1$. 
If $s=1$, we may assume $h_1=h_2=h_3=0$.  As in the proof of Theorem~\ref{quaternary_form_with_zero_hessian}
this would imply that there exist a linear relation among the partial derivatives of $f$. 
Since we have excluded this case,  we are left with the case  $s=2$. 
Then we may assume that  
$h_1=h_2=0$ and $h_j \in \sol(i(L(W)); R)$ for $j=3, 4, 5$. 
The linear subspace  $i(L(W))$  consists of vectors 
$(0,0, \ast, \ast, \ast)$.  Hence $h_3, h_4, h_5$ should be polynomials 
only in $x_1$ and $x_2$. 
\end{proof}

\begin{theorem} \label{thm45}
  Let $R=K[x_1, x_2, x_3, x_4, x_5]$, and let $\Delta$ be a homogeneous polynomial
of the form
$$
\Delta = p_3(x_1,x_2) x_3 + p_4(x_1,x_2) x_4 + p_5(x_1,x_2) x_5
$$
Then any element in the algebra \(K[x_1, x_2][\Delta]\) is a polynomial with zero Hessian.

Conversely, 
let $f$ be a homogeneous form in five variables with zero Hessian and assume that $f$ properly 
involves five variables. 
Then we can choose $\Delta$ such that $f$ can be transformed into a homogeneous polynomial 
in the algebra \(K[x_1, x_2][\Delta]\) by means of a linear change of variables.  
\end{theorem}

\begin{proof}
Put $\boldf=(f_1, \ldots, f_5)$, where $f_j=\frac{\pa f}{\pa x_j}$.  
Suppose first that $f \in K[x_1,x_2][\Delta]$.    Assume without loss of generality that
$f_5 \ne 0$. Then $f_3/f_5 = p_3/p_5 \in K(x_1/x_2)$ and 
$f_4/f_5 = p_4/p_5 \in K(x_1/x_2)$, so $f_3,f_4,f_5 \in K(x_1/x_2,f_5)$.
Hence $\trdeg_K K(f_3,f_4,f_5) \le 2$ and 
$\trdeg_K K(f_1,f_2,f_3,f_4,f_5) \le 4$.
On account of Proposition \ref{basic_1}, $f$ has zero Hessian.

Suppose next that $f$ has zero Hessian. On account of Proposition~\ref{basic_2},
$f_1,f_2,f_3,f_4,f_5$ are algebraically dependent over $K$.
Put $\boldf=(f_1, \ldots, f_5)$, where $f_j=\frac{\pa f}{\pa x_j}$.  
Let $\boldh=(h_1, \ldots, h_5)$ be a reduced 
self-vanishing system arising from $f(x)$.  We have proved that 
$\boldh \cdot \frac{\pa}{\pa x_j}\boldf=0,  j=1,\ldots, 5$ in Corollary~\ref{duality}. 
This shows that $f(x) \in \sol(\frac{\pa}{\pa x_1}\boldh, \ldots, \frac{\pa}{\pa x_5}\boldh; R)$. 
By Proposition~\ref{svs_for_forms_with_zero_hessian_in_5_variables}, we may assume that 
$h_1=h_2=0$ and $h_3, h_4, h_5$ involve only $x_1, x_2$. Let $\tilde{K}$ be the algebraic
closure of $K(x_1,x_2)$. It follows from
Theorem \ref{subring_by_generated_by_determinates} that $f$ is a polynomial
over $\tilde{K}$ in
$$
\tilde{A}' := 
\left| \begin{array}{ccc}
\frac{\pa h_3}{\pa x_1} & \frac{\pa h_4}{\pa x_1} & \frac{\pa h_5}{\pa x_1} \\[5pt]
\frac{\pa h_3}{\pa x_2} & \frac{\pa h_4}{\pa x_2} & \frac{\pa h_5}{\pa x_2} \\[4pt]
\hline
x_3 & x_4 & x_5
\end{array} \right|. 
$$
Notice that $\tilde{A}' \in R$ has only one irreducible factor which is \emph{not} 
contained in $K[x_1,x_2]$. Furthermore, we can choose $\Delta$ to be this factor.
Then $f \in \tilde{K}[\delta]$, say 
$$
f = b_0 + b_1 \Delta + b_2 \Delta^2 + \cdots .
$$
Then the coefficient of $x_j^i$ of $f$ as a polynomial in $K[x_1,x_2][x_3,x_4,x_5]$
equals $b_i p_j^i$, for $j=3,4,5$. So $b_i \in K(x_1,x_2)$.
Since ${\rm GCD}(p_3^i,p_4^i,p_5^i) = 1$, we infer that $b_i \in K[x_1,x_2]$.
This holds for all $i$, so $f \in K[x_1,x_2][\Delta]$.
\end{proof}

If $\Delta$ is cubic, then we can take $\Delta = A'$. But if $\Delta$ has larger degree,
then this is not always possible. Take e.g.\@ $\Delta$ irreducible of even degree, then 
we cannot take $\Delta = A'$, because $A'$ has always odd degree.
It is however possible to write $\Delta$ in the same way as $A'$, but with 
the $h^{(j)}_i$ replaced by other polynomials in $K[x_1,x_2]$.

\begin{theorem} \label{Delta_syzygy_determinant}
Let $n \geq 4$. 
Let $\Delta \in K[x_1,x_2,\ldots,x_n]$ be a homogeneous polynomial of 
degree at least $2$ of the form
$$
\Delta = p_3(x_1,x_2) x_3 + p_4(x_1,x_2) x_4 + \cdots + p_n(x_1,x_2) x_n
$$
Then there are polynomials $a^{(j)}_i$ in $K[x_1,x_2]$, 
such that $\bolda^{(j)} = (a^{(j)}_3,a^{(j)}_4,\ldots,a^{(j)}_n)$ 
is homogeneous for all $3 \le j \le n - 1$, and
$$
\Delta = \left| \begin{array}{cccc}
a^{(3)}_3 & a^{(3)}_4 & \cdots & a^{(3)}_n \\
a^{(4)}_3 & a^{(4)}_4 & \cdots & a^{(4)}_n \\
\vdots & \vdots & & \vdots \\
a^{(n-1)}_3 & a^{(n-1)}_4 & \cdots & a^{(n-1)}_n \\
\hline
x_3 & x_4 & \cdots & x_n
\end{array} \right|. 
$$ 
\end{theorem}

\begin{proof}
If theorem \ref{Delta_syzygy_determinant} holds for 
$\Delta\big(x_1,x_2,A(x_3,x_4,\ldots,x_n)\big)$ for some $A \in \GL(n-2,K)$
instead of $\Delta$ itself, then we can substitute 
$(x_3,x_4,\ldots,x_n)$ by $A^{-1}(x_3,x_4,\ldots,x_n)$ in the matrix with 
determinant $\Delta\big(x_1,x_2,A(x_3,x_4,\ldots,x_n)\big)$,
to obtain a matrix with determinant $\Delta$. Without affecting its determinant,
we can get the last matrix of the form of theorem \ref{Delta_syzygy_determinant} 
by way of column operations and multiplying the first row with a nonzero constant. 
This allows us to replace $\big(p_3(x_1,x_2),p_4(x_1,x_2),\ldots,p_n(x_1,x_2)\big)$
by $\big(p_3(x_1,x_2),
\ldots,p_n(x_1,x_2)\big)\,A$ 
for any $A \in \GL(n-2,K)$.

Suppose first that $p_3(x_1,x_2),p_4(x_1,x_2),\ldots,p_n(x_1,x_2)$
are linearly dependent over $K$. Then we may assume that 
$p_n(x_1,x_2) = 0$. If $n = 4$, then we take $a^{(3)}_3 = 0$ and 
$a^{(3)}_4 = -p_3(x_1,x_2)$. So assume that $n \ge 5$. Then we take
$a^{(n-1)}_i = 0$ for all $3 \le i < n$ and $a^{(n-1)}_n = -1$. 
Furthermore, we take $a^{(j)}_n = 0$ for all $3 \le j < n-1$.
By induction on $n$, there exist $a^{(j)}_i$ as claimed.

Suppose next that $p_3(x_1,x_2),p_4(x_1,x_2),\ldots,p_n(x_1,x_2)$
are linearly independent over $K$. Take $\alpha_3 \in K$, such that
$p_3(\alpha _3,1) \ne 0$, and assume without loss of generality that
$p_i(\alpha _3,1) = 0$ for all $i \ne 3$. Take $\alpha_4 \in K$, such that
$p_4(\alpha _4,1) \ne 0$, and assume without loss of generality that
$p_i(\alpha _4,1) = 0$ for all $i \ne 4$. Do the same with
$\alpha_5, \alpha_6, \ldots, \alpha_n$, and let
$$
\sigma := (x_1 - \alpha_3 x_2) (x_1 - \alpha_4 x_2) \cdots (x_1 - \alpha_n x_2). 
$$
Then
$$
\Delta' := \frac{(x_1 - \alpha_3 x_2) \,p_3}{\sigma} \,x_3 + 
\frac{(x_1 - \alpha_4 x_2) \,p_4}{\sigma} \,x_4 + \cdots + 
\frac{(x_1 - \alpha_n x_2) \,p_n}{\sigma} \,x_n
$$
is a polynomial, and by induction on the degree, we have
$$
\Delta' = \left| \begin{array}{cccc}
a^{(3)}_3 & a^{(3)}_4 & \cdots & a^{(3)}_n \\
a^{(4)}_3 & a^{(4)}_4 & \cdots & a^{(4)}_n \\
\vdots & \vdots & & \vdots \\
a^{(n-1)}_3 & a^{(n-1)}_4 & \cdots & a^{(n-1)}_n \\
\hline
x_3 & x_4 & \cdots & x_n
\end{array} \right|,
$$
with the above properties on the $a^{(j)}_i$. Consequently,
$$
\Delta = \left| \begin{array}{cccc}
(x_1-\alpha_3 x_2) \,a^{(3)}_3 & (x_1-\alpha_4 x_2) \,a^{(3)}_4 & \cdots & (x_1-\alpha_n x_2) \,a^{(3)}_n \\
(x_1-\alpha_3 x_2) \,a^{(4)}_3 & (x_1-\alpha_4 x_2) \,a^{(4)}_4 & \cdots & (x_1-\alpha_n x_2) \,a^{(4)}_n \\
\vdots & \vdots & & \vdots \\
(x_1-\alpha_3 x_2) \,a^{(n-1)}_3 & (x_1-\alpha_4 x_2) \,a^{(n-1)}_4 & \cdots & (x_1-\alpha_n x_2) \,a^{(n-1)}_n \\
\hline
x_3 & x_4 & \cdots & x_n
\end{array} \right|.
$$
This is essentially equivalent to Hilbert-Burch Theorem. The interested reader may wish to see \cite{eisenbud}~Theorem~20.15.
\end{proof}

\begin{remark} \label{rem_by_H-Nasu} 
Hirokazu Nasu showed  that the variety $X=(\Delta =0) \subset \PP ^4$ is isomorphic to 
the projection of the Segre variety $S \subset  \PP ^5$ of degree three from a general 
point outside of $S$. Here $S \subset \PP^5$ is defined as the image of the Segre embedding 
\[   \PP ^1 \times \PP ^2 \hookrightarrow \PP ^5. \]
This is the locus of the maximal minors of a generic 
$2 \times  3$ matrix. 

It is easy to see that  any degree three homogeneous polynomial $F \in K[x_1,x_2,\Delta]$ in 
Theorem~\ref{thm45} 
which properly involves five variables
can be transformed into the canonical form  $x_1^2\,x_3+ x_1x_2\,x_4+ x_2^2\,x_5$ by means of a 
linear transformation of the variables.  
This form  is also known as the Macaulay dual of the 
trivial extension of the algebra $K[x_1, x_2]/ (x_1,x_2)^3$ by the canonical module.  
\end{remark} 


\section{Proof of {\mathversion{bold}$\dim \; W \leq 1$ in the proof of 
Proposition \ref{svs_for_forms_with_zero_hessian_in_5_variables}}}

We first formulate a result about homogeneous self-vanishing systems in dimension $5$
in general.

\begin{theorem} \label{h}
Let $\bm{h} = (h_1,h_2,h_3,h_4,h_5)$ be a homogeneous self-vanishing system,
for which 
$$
\rk \Big(\frac{\partial h_i}{\partial x_j}\Big) = \trdeg_K K(\bm{h}) = \krdim K[\bm{h}] \ge 3.
$$
Let $Z: \PP^4 \dashrightarrow \PP^4$ be the rational map, defined by
$$
x = (x_1 : \cdots : x_5) \mapsto (h_1 : \cdots : h_5),
$$
and let $W$ be the closure of the image of $Z$.

If $\dim\, L(W) > 1$, then one of the following holds.
\begin{enumerate}

\item[\upshape(1)] $\dim\, L(W) = 3$ (where $L(W)$ is the $K$-linear span of $W$ as a variety in $\PP^4$).

\item[\upshape(2)] $W$ has a vertex, i.e., a point $p \in \PP^4$ such that $W$ is a union
of lines through $p$.

\end{enumerate}
\end{theorem}

Using Theorem \ref{h}, we can prove that $\dim W \le 1$. Cases (1) and (2) are 
covered by Lemmas \ref{A} and \ref{B} respectively. Lemma \ref{A} and its 
proof are essentially from \cite[p.\@ 567]{GNzzz}. Lemma \ref{B} is proved on 
\cite[p.\@ 568]{GNzzz}, and we think the proof is correct, but it could use some 
justification. For that reason, we formulated an alternative proof. Both 
Lemma \ref{A} and Lemma \ref{B} come from the second author's paper \cite{hmgqt5}.

\begin{lemma} \label{A}
Suppose that $f \in K[x_1,x_2,x_3,x_4,x_5]$ is homogeneous with zero Hessian, 
and let $\bm{h}$ be a reduced self-vanishing system arising from $f$. Then 
$\bm{h}$ is not as in {\upshape(1)} of Theorem \ref{h}.
\end{lemma}

\begin{proof}
Suppose that $\bm{h}$ is as in {\upshape(1)} of Theorem \ref{h}. Then
$\dim\, L(W) = 3$, so we may assume that the last coordinate of every point of 
$W$ is zero. Develop $f$ into powers of $x_5$, and write
$$
f = x_5^{\mu} A + x_5^{\mu+1} B,
$$
where $A \in K[x_1,x_2,x_3,x_4]$  and $B \in K[x_1,x_2,x_3,x_4,x_5]$.
Take $g$ as in Definition 4.4. Then
$$
g\Big(\parder{f}{x_1},\parder{f}{x_2},\parder{f}{x_3},\parder{f}{x_4},\parder{f}{x_5}\Big) = 0.
$$
Since $g$ has minimum degree as such, $\parder{g}{y_5} = 0$. 
Hence $g \in K[y_1,y_2,y_3,y_4]$. If we look at the trailing coefficient 
with respect to $x_5$, we see that 
\begin{equation} \label{gA}
g\Big(\parder{A}{x_1},\parder{A}{x_2},\parder{A}{x_3},\parder{A}{x_4}\Big), 
\end{equation}
i.e., $A \in K[x_1,x_2,x_3,x_4]$ has Hessian determinant zero (in dimension $4$). 
From Theorem \ref{quaternary_form_with_zero_hessian}, it follows that we may assume 
that $A \in K[x_1,x_2,x_3]$.

If $y_4 \mid g$, then $g = y_4$ because $g$ has minimum degree, so $\deg\, \bm{h} = 0$.
Consequently, $y_4 \nmid g$.
It follows from $A \in K[x_1,x_2,x_3]$ and 
\eqref{gA} that $A \in K[x_1,x_2,x_3]$ has Hessian determinant zero (in dimension $3$).
From Theorem \ref{ternary_form_with_zero_hessian_is_trivial}, 
it follows that we may assume that $A \in K[x_1,x_2]$.

Using Theorem \ref{duality} (a)  
and Theorem \ref{mainthm_1}, 
we deduce that 
$f\big(x +t\,\bm{h}(x)\big) = f(x)$. As $h_5 = 0$, we have
$$
A\big(x + t\,\bm{h}(x)\big) x_5^{\mu} + B\big(x + t\,\bm{h}(x)\big) x_5^{\mu+1} 
= x_5^{\mu} A(x) + x_5^{\mu+1} B(x).
$$
If we look at the leading coefficient with respect to $t$, we see that 
$A(\bm{h}(x)) = 0$. Since $A \in K[x_1,x_2]$, we see that $h_1$ and $h_2$ are algebraically
dependent over $K$. Consequently, $h_1$ and $h_2$ are linearly dependent
over $K$, say that $h_1 = 0$. Then $h_1 = h_5 = 0$, so $\dim\, L(W) \le 2$. 
Contradiction.
\end{proof}

\begin{lemma} \label{B}
Suppose that $f \in K[x] = K[x_1,x_2,\ldots,x_n]$ is homogeneous with zero Hessian, 
and let $\bm{h}$ be a reduced self-vanishing system arising from $f$. Suppose that 
$p$ is a vertex of $W$. Then $f \in \sol(p; K[x])$. In particular, a variable can be 
eliminated from $f$ by way of a linear transformation of the variables.
\end{lemma}

\begin{proof}
Write $p = (p_1 : p_2 : \cdots : p_n)$ and take $d := \deg\, \boldh$.
Take $\tilde{\boldh} = \boldh + x_{n+1}^d p$. Since $p$ is a vertex of $W$, 
we infer that the image of 
$$
\tilde{Z} : \PP^n(x) \rightarrow \PP^{n-1}(y),
$$
defined by $(x_1 : x_2 : \cdots : x_n : x_{n+1}) \mapsto 
(\tilde{h}_1 : \tilde{h}_2 : \cdots : \tilde{h}_n)$, is equal to $W$.
Hence it follows from Lemma~\ref{rank_of_hessian_matrix_corrected} that  
$$
\rk\, (\pa \tilde{h}_i / \pa x_j) = \krdim K[\tilde{\boldh}] = 
\dim\, W + 1 = \krdim K[\boldh] = \rk\, (\pa h_i / \pa x_j).
$$
Since $(\pa \tilde{h}_i / \pa x_j)$ can be obtained from $(\pa h_i / \pa x_j)$
by adding $\pa \tilde{\boldh} / \pa x_{n+1}$ as a column to the right hand side,
we infer from $\rk\, (\pa \tilde{h}_i / \pa x_j) = \rk\, (\pa h_i / \pa x_j)$ that
$\pa \tilde{\boldh} / \pa x_{n+1} = d\,x_{n+1}^{d-1}\,p$ is contained in the column 
space of $(\pa h_i / \pa x_j)$. 

Hence $p$ is dependent over $K(x)$ on
$\pa \boldh / \pa x_1, \pa \boldh / \pa x_2, \ldots, \pa \boldh / \pa x_n$. 
From Theorem \ref{duality} (c), we infer that $f \in \sol(p; K[x])$.
\end{proof}

\emph{The rest of this section will be devoted to the proof of Theorem \ref{h}.}

Suppose that $\dim\, W > 1$, and let $T$ be the fundamental locus of $Z$ in $\PP^4(x)$, 
defined by the equations $h_1(x) = h_2(x) = \cdots = h_5(x) = 0$. 
From Theorem \ref{mainthm_1} (c), it follows that $\bm{h}(x+t\bm{h}(x)) = \bm{h}(x)$. 
If we look at the leading coefficient with respect to $t$, we see that 
\begin{equation} \label{hoh}
\bm{h}(\bm{h}(x)) = 0, \qquad \mbox{i.e.,} \qquad W \subseteq T.
\end{equation}
From Proposition \ref{dimension_of_T}, it follows that $\dim\, T \le 2$. Since 
$\dim\, W \ge 2$, \eqref{hoh} tells us that 
$$
\dim\, W = \dim\, T = 2.
$$
As $W$ is the closure of the image of $Z$, we see that $W$ is a component of $T$.

On \cite[p.\@ 565]{GNzzz}, the authors claim that for every $c \in T$, there exists a
$p \in W$ such that the line through $c$ and $p$ is contained in $T$. 
But if $c = p$, then `the line through $c$ and $p$' shrinks to a single point, so
it must be shown that $p$ can be taken different from $c$. The following lemma
can be used for that.

\begin{lemma} \label{X}
Let $\bm{h} = (h_1, h_2, \ldots,h_n)$ be a homogeneous self-vanishing system. 
Let $S$ be a hyperplane in $\PP^{n-1}$. Then every irreducible component of $Z^{-1}(S)$
contains $W$ and has dimension $n-2$.
\end{lemma}

\begin{proof}
We can take $\bm{\alpha} = (\alpha_1,\alpha_2,\ldots,\alpha_n)$,
such that
$$
S = \big\{\,(\sigma_1 : \sigma_2 : \cdots : \sigma_n)\,\big|\,
\alpha_1 \sigma_1 + \alpha_2 \sigma_2 + \cdots + \alpha_n \sigma_n = 0\,\}.
$$
Now $Z^{-1}(S)$ is contained in
$$
\big\{\,(\tau_1 : \tau_2 : \cdots : \tau_n)\,\big|\,
\alpha_1 h_1(\bm{\tau}) + \alpha_2 h_2(\bm{\tau}) + \cdots + \alpha_n h_n(\bm{\tau}) = 0\,\},
$$
where $\bm{\tau} = (\tau_1,\tau_2,\ldots,\tau_n)$, and any irreducible component $X$ of $Z^{-1}(S)$ is of the form 
$$
\big\{\,(\tau_1 : \tau_2 : \cdots : \tau_n)\,\big|\,f(\bm{\tau}) = 0\,\},
$$
where $f$ is an irreducible factor of 
$\alpha_1 h_1 + \alpha_2 h_2 + \cdots + \alpha_n h_n$. So $\dim\, X = n-2$.

It suffices to show that $f(p) = 0$ for every $p$ in the image of $Z$
(i.e., skip the points of $W$ which were added by taking closure).
So let $p$ be an image point of $W$. From Lemma 
\ref{rank_of_hessian_matrix_corrected}, it follows
that
\[
\rk \Big(\parder{h_i}{x_j}\Big) = \krdim K[\bm{h}] \le n-1.
\]
On account of the fiber dimension theorem, the closure of $Z^{-1}(p)$ 
has dimension at least $1$. As $\dim\, X = n-2$,
the intersection of $X$ and the closure of $Z^{-1}(p)$ is nonempty, say that
$\theta$ is contained in this intersection.

If $\theta \in Z^{-1}(p)$, then $p=\bm{h}(\theta)$, and from Theorem~\ref{mainthm_1}
we infer that $\bm{h}(\bm{\theta}+t\bm{p}) = \bm{h}(\bm{\theta})$, where $\bm{\theta}$ and 
$\bm{p}$ are vectors over $K$ which correspond to $\theta$ and $p$ respectively. 
The last equality holds in general as well because $\theta$ is contained in the closure of $Z^{-1}(p)$.  
Hence
$$
\deg_t f(\bm{\theta}+t\bm{p}) \le 
\deg_t \big(\alpha_1 h_1(\bm{\theta}+t\bm{p}) + \alpha_2 h_2(\bm{\theta}+t\bm{p}) + \cdots + 
\alpha_n h_n(\bm{\theta}+t\bm{p})\big) = 0.
$$
So $f(\bm{\theta}+t\bm{p}) = f(\bm{\theta})$. From $\theta \in X$, it follows that 
$f(\bm{\theta}+t\bm{p}) = f(\bm{\theta}) = 0$. If we look at the leading coefficient 
with respect to $t$, we see that $f(\bm{p}) = 0$, which completes the proof.
\end{proof}

\begin{corollary} \label{Y}
Let $n \ge 3$ and $\bm{h} = (h_1, h_2, \ldots,h_n)$ be a homogeneous self-vanishing system.
Let $S$ be a hyperplane in $\PP^{n-1}$, and $Y$ be a component of $S \cap W$.
Assume that $Y$ contains an image point of $Z$. 
Then for every $c \in W$, there exists a $p \in Y$ such that the line
through $c$ and $p$ is contained in $T$.
\end{corollary}

\begin{proof}
  Let $X$ be any component of the closure of $Z^{-1}(Y)$.
  For every $c \in X$
for which $\bm{h}(c) \ne 0$, there exists a $p \in Y$ such that 
$\bm{h}(\bm{c}+t\bm{p}) = \bm{h}(\bm{c})$, where $\bm{c}$ and $\bm{p}$ are
vectors over $K$ which correspond to $c$ and $p$ respectively. 
To see this just set $p = \bm{h}(c)$ and use Theorem~\ref{mainthm_1}. 

Let
$$
U = \{ (c,p) \in X \times Y | \bm{h}(\bm{c}+t\bm{p}) = \bm{h}(\bm{c}) \}. 
$$
Since $Y$ is a complete variety, the projection of $U$ on $X$ is a closed
morphism. Hence
$$
\tilde{X} = \{ c \in X |\,\mbox{there is a $p \in Y$ such that 
$\bm{h}(\bm{c}+t\bm{p}) = \bm{h}(\bm{c})$}\,\}
$$ is closed. The $c \in X$
for which $\bm{h}(c) \ne 0$ form a dense subset of $X$ and are contained in 
$\tilde{X}$, so $\tilde{X} = X$. 

From lemma \ref{X}, it follows that $X$ contains $W$, so 
$W \subseteq \tilde{X}$. So for every $c \in W$, there exists a 
$p \in Y$ such that 
$\bm{h}(\bm{c}+t\bm{p}) = \bm{h}(\bm{c}) = 0$, i.e.,
the line through $c$ and $p$ is contained in $T$.
\end{proof}

So let us take $c \in W$. It is possible to take  $S$ such that $c \notin S$ and $S$ contains an image point of $Z$.
From corollary \ref{Y}, it follows that there exist a $c' \in S$ such that  
the line through $c$ and $c'$ is contained in $T$. As $c \notin S$ and $c' \in S$, 
we see that `the line through $c$ and $c'$' does not shrink to a single point.

Since $W$ is a component of $T$, the interior $W^{\circ}$ of $W$
as a subspace of $T$ is nonempty. Now take $p \in W^{\circ}$ in the image of $Z$.
There exists a $p' \in W$ such that $p' \ne p$ and such that the line $L_p$ through $p$
and $p'$ is contained in $T$. But since $W$ is the only component of $T$ which contains
$p$, it follows that $L_p \subseteq W$. 

Taking $Y = L_p$ in corollary \ref{Y}, we deduce that for every $q \in W$,
there exists a $q' \in L_p$ such that the line $L_q$ through $q$ and $q'$ 
is contained in $T$. This is also claimed on \cite[p.\@ 565]{GNzzz}. If we take
$q \in W^{\circ}$, then $L_q$ is even contained in $W$.

Now let us fix $L_p$, and range $q$ over $W^{\circ}$. There are infinitely
many points $q \in W^{\circ}$, but this does not mean automatically that 
there are infinitely many lines $L_q$. This is because for a line $L_q$, 
there may be infinitely many candidates for the point $q$. 

However, since $\dim\, W = 2$, there are infinitely many lines $L_q$ indeed,
just as claimed on \cite[p.\@ 565]{GNzzz}. On \cite[pp.\@ 565, 566]{GNzzz}, the following two 
cases (a) and (b) are distinguished:
\begin{enumerate}

\item[(a)] The set of points $q' \in L_p$ (which we get by ranging $q$ over
$W^{\circ}$) is infinite.

\item[(b)] There exists a fixed point $q' \in L_p$, which is contained in infinitely
many lines $L_q$.

\end{enumerate}
We will treat these cases in essentially the same way as on \cite[pp.\@ 566, 568]{GNzzz}.

Assume first that case (b) above applies. 
Then one can show that the closure of the union of lines $L_q$ through $q'$ 
has dimension $2$. As $W$ is irreducible of dimension $2$,
$W$ is just the closure of this union, so $W$ is a union of lines through $q'$. 
Thus case (b) corresponds to case (2) of Theorem \ref{h}, as claimed on \cite[p.\@ 568]{GNzzz}.

Assume next that case (a) above applies. Take $r \in W^{\circ}$. Then there exists an $r' \in L_p$ 
such that the line $L_r$ through $r$ and $r'$ is contained in $W$. 

Now let us fix $L_q$ as well, by choosing $q$ in $W^{\circ}$ in the image of $Z$ outside $L_p$.  Range $r$ over $W^{\circ}$. 
On \cite[p.\@ 566]{GNzzz}, two subcases of case (a) are distinguished, which are essentially as follows:
\begin{enumerate}

\item[(a1)] There does not exist a line $L_r$ as above, such that $L_r \cap L_q = \varnothing$.

\item[(a2)] There does exist a line $L_r$ as above, such that $L_r \cap L_q = \varnothing$.

\end{enumerate}
Both cases are treated essentially as follows on \cite[p.\@ 566]{GNzzz}.

Assume first that case (a1) above applies. 
Then one can show that the linear span of $p$, $p'$ and $q$ contains
infinitely many lines $L_r$, and that the closure of these lines has dimension $2$.
As $W$ is irreducible of dimension $2$, $W$ is just the closure of these lines, 
which corresponds to the linear span of $p$, $p'$ and $q$. 

So $\dim\, L(W) \le 2$. Hence $\dim\, L(W) = \dim\, W$, so $L(W) = W$ and every
point of $W$ is a vertex of $W$. So $W$ is as in (2) of Theorem \ref{h}
(and case (b) above applies as well). 

Assume next that case (a2) above applies. Take $r \in W^{\circ}$ and $r' \in L_p$, 
such that $L_r \cap L_q = \varnothing$. Take $s \in L_r \cap W^{\circ}$. 
Just as with $r' \in L_p$, there exists an $s' \in L_q$, such that the line
$L_s$ through $s$ and $s'$ is contained in $W$. 

There are infinitely many lines $L_s$, if we range 
$s$ over $L_r \cap W^{\circ}$. Let $U$ be the closure of the lines $L_s$.
Only finitely many lines can be a component of $U$, so there is a line which is not. 
Being irreducible, the line is fully contained in a component of $U$, and this 
component has larger dimension than the line. So $U$ has dimension at least $2$.
Consequently, the intersection of $W$ and the linear span
of $q$, $q'$, $r$ and $r'$ has dimension at least $2$. As $W$ is irreducible
of dimension $2$, $W$ is just this intersection. So $W$ is as in (2) of
Theorem \ref{h}.

\begin{remark}
Self-vanishing systems corresponding to cases (a1), (a2), (b) indeed exist:
\begin{enumerate}

\item[(a1)] $H = (x_4^2,x_4x_5,x_1x_5-x_2x_4,0,0)$,

\item[(a2)] $H = (x_5^2(ax_1-x_5^2x_2),a(ax_1-x_5^2x_2),x_5^2(ax_3-x_5^2x_4),
            a(ax_3-x_5^2x_4),0)$ with $a = x_1x_4-x_2x_3$,

\item[(b)]  $H = (x_5^5,bx_5^3,b^2x_5,-b^2x_1+2bx_2x_5^2-x_3x_5^4,0)$ with 
            $b = x_1x_3-x_2^2+x_4x_5$.

\end{enumerate}
All these examples are counterexamples to the original version of Lemma 
\ref{rank_of_hessian_matrix_corrected}. They were taken from the introduction 
of the second author's paper \cite{hmgqt5}. 
\end{remark}

\begin{remark}
Gordan and Noether only prove that $\dim\, L(W) \le 3$ in Theorem \ref{h} (1).
This makes that the conclusion that $\dim\, L(W) \le 2$ at the end of of the 
proof of Lemma \ref{A} is not sufficient to prove Lemma \ref{A}. 
Gordan and Noether advance as follows. 

On account of $\dim\, L(W) \le 2$, we may assume that $h_1 = h_2 = 0$. If 
$f \in K[x_1,x_2]$, then $h_i \in K[x_1,x_2]$ for all $i$ as well, so 
$\dim\, W = \krdim K[\bm{h}] - 1 = \trdeg_K K(\bm{h}) - 1 \le 1$.
If $f \notin K[x_1,x_2]$, then it follows from Theorem \ref{duality} (c) that
the last three rows of the Jacobian matrix $(\pa h_i/\pa x_j)$ of $\bm{h}$
are dependent. The first two rows are zero, so $\rk (\pa h_i / \pa x_j) \le 2$ 
and $\dim\, W = \krdim K[\bm{h}] - 1 = \rk (\pa h_i / \pa x_j)  - 1 \le 1$.
\end{remark}

\begin{remark} \label{linspanrem}
Theorem \ref{h} is proved in \cite{hmgqt5} by the second author as well, but in a 
different way, because the proof by Gordan and Noether was not fully understood.
The second author only proved that $\dim\, L(W) \le 2$ in \cite{hmgqt5}, because
that was the missing link in \cite{singhess} to obtain Theorem \ref{thm45}.

The focus in on $\dim\, L(W)$ rather than $\dim\, W$ in other papers of 
the second author as well.
In \cite{qtnsh}, only $\dim\, L(W) \le 1$ is used in the classification of all 
homogeneous polynomials with zero Hessian in dimension $4$. 
In \cite{hessmalrank}, all homogeneous polynomials with zero Hessian in 
dimension $6$ are classified, under the assumption that $\dim\, L(W) \le 3$.
Non-homogeneous polynomials are classified under similar assumptions in
\cite{singhess} and \cite{hessmalrank}.
\end{remark}

\section{Section 6 of Gordan and Noether~\cite{GNzzz}} 

We start with formulating a result which, as opposed to 
Theorem \ref{duality}, applies to $\boldh$ which are not constructed as in 
Proposition \ref{system_associated_to_alg_dep_system_is_svs}. 
The result is inspired by Lemma 3.1 of \cite{clossen}.

\begin{proposition} \label{ab}
Let $\boldh = (h_1,h_2,\ldots,h_n)$ be a system of forms of the same degree, 
and let $f \in K[x]$ be a homogeneous polynomial. Put $f_j=\frac{\pa f}{\pa x_j}$
for $j =1,2,\ldots,n$, and let $\boldf=(f_1, \dots, f_n)$
Then the following statements are equivalent:
\begin{enumerate}
\item[{\rm (a)}] 
$\pa \boldh/\pa x_j$ is a syzygy of $\boldf$ for $j=1,2, \ldots, n$.  
\item[{\rm (b)}]
$\boldh$ is a syzygy of $\pa \boldf/\pa x_j$ for $j=1,2, \ldots, n$.
\end{enumerate}
Furthermore, both {\rm (a)} and {\rm (b)} imply that $\boldh$ is a syzygy of 
$\boldf$.
\end{proposition}

\begin{proof}
Using
$$
\frac{\pa \boldh}{\pa x_1} x_1 + \frac{\pa \boldh}{\pa x_2} x_2 + 
\cdots + \frac{\pa \boldh}{\pa x_n} x_n = (\deg\; \boldh) \boldh
$$
we can obtain the last claim from (a). Similarly, the last claim can
be obtained from (b). Having these results, the equivalence of (a) and (b)
follows from the formula 
\[  \frac{\pa}{\pa x_j}(\boldh\cdot \boldf)= \frac{\pa \boldh}{\pa x_j} \cdot \boldf + \boldh \cdot \frac{\pa \boldf}{\pa x_j}=0.\]
for $j = 1, 2, \ldots, n$. This formula appears in the proof of Theorem \ref{duality}.
\end{proof}

Let $f \in K[x]$ be a homogeneous polynomial, and put $f_j = \frac{\pa f}{\pa x_j}$.
Assume that $f_1, f_2, \ldots, f_n$ are algebraically dependent over $K$. Then
by Proposition \ref{basic_2}, 
there exists a system $\boldh = (h_1,h_2,\ldots,h_n)$ of forms of the same degree,
such that Proposition \ref{ab} (b) is satisfied. 

Assume that 
\begin{equation}  \label{splitting_case_2}
h_j(x) \in \sol(i(L(W));R) \ \mbox{ for all } j. 
\end{equation}
Then by a linear change of variables, $\boldh$ coincides with a self-vanishing system described in 
Example~\ref{splitting_case}.
We assume this from now on, so
\begin{itemize}

\item $h_1(x) = \cdots = h_r(x) = 0$, for some $ 1  \leq r < n$, and 

\item $h_j(x)$ is a function only in $x_1, \ldots, x_r$ for $ j > r$.

\end{itemize}
In \S 6 of \cite{GNzzz} and \S 3 of \cite{clossen}, all $f$ which satisfy 
Proposition \ref{ab} (b) are classified for this particular $\boldh$.

Although $\boldh$ is a self-vanishing system which has many properties of 
reduced self-vanishing systems arising from $f$, especially if 
${\rm GCD}(h_1,h_2,\ldots,h_n) = 1$, we will show in Example \ref{is_not_arising} 
below that it is possible for $\boldh$ to satisfy ${\rm GCD}(h_1,h_2,\ldots,h_n) = 1$
and not to be of this form for some $f$ which satisfies Proposition \ref{ab} (b).

But if ${\rm GCD}(h_1,h_2,\ldots,h_n) = 1$ and $\cI(f)$ is a principal ideal, 
then $\boldh$ is indeed a reduced self-vanishing systems of $f$. This is because
there is only one $\boldh$ with ${\rm GCD}(h_1,h_2,\ldots,h_n) = 1$ which 
satisfy Proposition \ref{ab} (a).

\begin{theorem} \label{thm92}
Suppose that the Jacobian $M := (\pa h_i/\pa x_j)$ of $\boldh = (h_1,h_2,\ldots,h_n)$
has rank $k$. Then we can choose $k$ columns of $M$ which generate its columns space,
say with indices $i_1, i_2, \ldots, i_k$. Then $i_1, i_2, \ldots, i_k \in \{1,2,\ldots,r\}$. 

Suppose that $f \in K(x_1,x_2,\ldots,x_r)[x_{r+1},x_{r+2},\ldots,x_n]$.
Then $f$ is a polynomial over $K(x_1,x_2,\ldots,x_r)$ in the $(k+1) \times (k+1)$
minors of 
$$
\left( \begin{array}{cccc}
\parder{h_{r+1}}{x_{i_1}} & \parder{h_{r+2}}{x_{i_1}} & \cdots & \parder{h_n}{x_{i_1}} \\[5pt]
\parder{h_{r+1}}{x_{i_2}} & \parder{h_{r+2}}{x_{i_2}} & \cdots & \parder{h_n}{x_{i_2}} \\[2pt]
\vdots & \vdots & & \vdots \\[2pt]
\parder{h_{r+1}}{x_{i_k}} & \parder{h_{r+2}}{x_{i_k}} & \cdots & \parder{h_n}{x_{i_k}} \\[5pt]
\hline
x_{r+1} & x_{x+2} & \cdots & x_n
\end{array} \right)
$$
if and only if $f$ satisfies Proposition \ref{ab} (b). 
\end{theorem}

\begin{proof}
Let $\tilde{K}$ be the algebraic closure of $K(x_1,x_2,\ldots,x_r)$.
From Theorem \ref{subring_by_generated_by_determinates}, it follows that 
$f$ is a polynomial over $\tilde{K}$ in the above-described minors, 
if and only if $f$ satisfies Proposition \ref{ab} (a). 

Suppose that $f$ is a polynomial over $\tilde{K}$ in the above-described 
minors. Let $\cB$ be a $K(x_1,x_2,\ldots,x_r)$-basis of $\tilde{K}$, 
such that $1 \in \cB$.  Then $\cB$ is also a basis of 
$\tilde{K}[x_{r+1},x_{r+2},\ldots,x_n]$ as a free module over 
$K(x_1,x_2,\ldots,x_r)[x_{r+1},x_{r+2},\ldots,x_n]$. Taking
coefficients of $1 \in \cB$ yields $f$ as a polynomial over 
$K(x_1,x_2,\allowbreak\ldots,x_r)$ in the above-described minors. 
\end{proof}

\begin{remark}
The proof of Theorem \ref{subring_by_generated_by_determinates} tells us that
the $(k+1) \times (k+1)$-minors can be replaced by $n-r-k$ homogeneous 
polynomials $l_1, l_2, \ldots, l_{n-r-k}$ which are linear in 
$x_{r+1}, x_{r+2}, \ldots, x_n$. 

In \S 2.3 of \cite{ciliberto_russo_francesco_simis}, a similar construction is
given as above, but some conditions on that construction makes it incomplete
for classification. This is however not a real problem, because there 
is no classification result in \cite{ciliberto_russo_francesco_simis}.
There is just a definition of so-called GN-polynomials of type $(r,s,\mu,n)$.

Let us call a polynomial $f$ as in Theorem \ref{thm92} a 
\emph{GN-polynomial of type $(n-1,n-1-r,k-1)$}. Then GN-polynomials of type 
$(n-1,n-1-r,k-1,n')$, as defined in \cite{ciliberto_russo_francesco_simis},
are GN-polynomials of type $(n-1,n-1-r,k-1)$ with additional conditions.
\end{remark}

\begin{remark}
Since $\trdeg_K K(h_{r+1},h_{r+2},\ldots,h_r) = k$, there 
exists a transcendence basis $A_1, A_2, \ldots, A_k$ over $K$ of 
$K(h_{r+1},h_{r+2},\ldots,h_r)$.
In \S 6 of \cite{GNzzz}, \S 3 of \cite{clossen}, and \S 2.3 of 
\cite{ciliberto_russo_francesco_simis}, derivatives are taken 
with respect to $A_1, A_2, \ldots, A_k$ instead of 
$x_{i_1}, x_{i_2}, \ldots, x_{i_k}$ in the matrix of Theorem
\ref{thm92}.

For other choices of $A_1, A_2, \ldots, A_k$, there are
problems with the meaning of differentiating $h_j$ with respect to $A_i$. 
To obtain meaning, we choose $A_{k+1}, \allowbreak A_{k+2}, \ldots, A_r$, such that 
$A_1, A_2, \ldots, A_r$ becomes a transcendence basis of $K(x_1,x_2,\ldots,x_r)$.
Now the only condition is that the first $k$ rows of the matrix of Theorem
\ref{thm92} are independent. With this condition, it is possible to take
for $A_1, A_2, \ldots, A_r$ a permutation of $x_1,x_2,\ldots,x_r$, which is
exactly what we do in Theorem \ref{thm92}.
\end{remark}

\begin{example} \label{is_not_arising}
If $\cI(f)$ is not a principal ideal, then is possible for $\boldh$ to satisfy 
Proposition \ref{ab} (b) and ${\rm GCD}(h_1,h_2,\ldots,h_n) = 1$ without being 
a reduced self-vanishing systems arising from $f$. Take for instance
$$
f = x_1^2 x_3 + x_1 x_2 x_4 + x_2^2 x_5 + z_1^2 z_3 + z_1 z_2 z_4 + z_2^2 z_5, 
$$
where $z_1,z_2,z_3,z_4,z_5 = x_6,x_7,x_8,x_9,x_{10}$. 
Then for a reduced self-vanishing system arising from $f$, the degree 
which respect to $x_1,x_2,x_3,x_4,x_5$ has the same parity as the 
degree with respect to $z_1,z_2,z_3,z_4,z_5$. This is however not the case 
for 
$$
\boldh = \big(0,0,(x_1 \cdot x_2^2), (x_1 \cdot -2 x_1 x_2), (x_1 \cdot x_1^2),
0,0,(x_2 \cdot z_2^2), (x_2 \cdot -2 z_1 z_2), (x_2 \cdot z_1^2)\big). 
$$
\end{example}

\end{document}